\documentclass[11pt,a4paper]{amsart}

\usepackage{a4wide}
\usepackage[utf8]{inputenc}
\usepackage{amsmath}
\usepackage{amsthm}
\usepackage{mathrsfs}
\usepackage{wasysym}
\usepackage{enumitem}
\usepackage{cite}
\usepackage[footnotesize]{caption}
\usepackage{tgschola,eulervm}
\usepackage{comment}
\usepackage{amsfonts}
\usepackage{amssymb}
\usepackage[final]{hyperref}

\allowdisplaybreaks
\makeatletter
\newcommand{\mainsectionstyle}{%
  \renewcommand{\@secnumfont}{\bfseries}
  \renewcommand\section{\@startsection{section}{2}%
    \z@{.5\linespacing\@plus.7\linespacing}{-.5em}%
    {\normalfont\bfseries}}%
}
\makeatother

\DeclareMathOperator{\dist}{dist}

\newcommand{\R}{\mathbb{R}}
\newcommand{\N}{\mathbb{N}}

\renewcommand{\S}{\mathbb{S}}
\newcommand{\Z}{\mathbb{Z}}

\newcommand{\HM}{\mathcal{H}}

\newtheorem{thm}{Theorem}[section]
\newtheorem*{thm*}{Theorem}

\newtheorem{lem}[thm]{Lemma}

\newtheorem*{introthm*}{Regularity Theorem}

\theoremstyle{definition}
\newtheorem{defin}[thm]{Definition}

\newtheorem{rem}[thm]{Remark}
\newtheorem*{rem*}{Remark}
\newtheorem{ex}[thm]{Example}

\newcommand{\Fo}{\,\,\,\text{for }\,\,}
\newcommand{\Foa}{\,\,\,\text{for all }\,\,}

\newcommand{\With}{\,\,\,\text{with }\,\,}

\newcommand{\AND}{\,\,\,\text{and }\,\,}

\author{Bastian K{\"a}fer}
\address[B.~K{\"a}fer]{
 \newline{}
 Institut f{\"u}r Mathematik
 \newline{}
 RWTH Aachen University
 \newline{}
 Templergraben~55
 \newline{}
 D-52062 Aachen, Germany}
\email{kaefer@instmath.rwth-aachen.de}


\date{\today}

\title[Reifenberg type characterization for $C^1$-submanifolds
]
{A Reifenberg type characterization for $m$-dimensional $C^1$-submanifolds of $\R^n$
  }

\begin{document}
\frenchspacing

\begin{abstract}
We provide a Reifenberg type characterization for $m$-dimensional $C^1$-submani-folds of $\R^n$. This characterization is also equivalent to Reifenberg-flatness with vanishing constant combined with suitably converging approximating $m$-planes. 
Moreover, a sufficient condition can be given by the finiteness of the integral of the quotient of $\theta(r)$-numbers and the scale $r$, and examples are presented to show that this last condition is not necessary. 
\end{abstract}

\maketitle
%%%%%%%%%%%%%%%%%%%%%%%%%%%%%%%%%%%%%%%%%%%%%%%%%%%%%%%%%%%%%%%%%%%%%%%%%%%%%%%%%%%%%%%%%%%%%%%%%%

\section{\bf{Introduction}} \label{sec:int}
It is often useful to control local geometric properties of a subset $\Sigma \subset \R^n$ to obtain topological and analytical information about that set.
One of these geometric properties is the local flatness of a set, first introduced and studied by E. R. Reifenberg in \cite{reifenberg_1960} for his solution of the Plateau problem in arbitrary dimensions. The content of his so-called Topological-Disk Theorem is that $\delta$-Reifenberg-flatness ensures that 
$\Sigma$ is locally a topological $C^{0,\alpha}$-disk if $\delta<\delta_0$, where $\delta_0=\delta_0(m,n)$ is a positive constant, which depends only on the dimensions of $\Sigma$ and $n$ (see e.g. \cite{reifenberg_1960}, \cite{morrey_1966}, \cite{hong-wang_2010}). 

\begin{defin} \label{def:Reifenberg-flat}
 Let $n,m \in \N$ with $m<n$ and $\Sigma \subset \R^n$. For $x \in \Sigma$ and $r>0$ set
  \[ \theta_\Sigma(x,r) := \frac 1 r \inf_{L \in G(n,m)} \dist_\HM \Big( \Sigma \cap B_r(x), (x+L) \cap B_r(x) \Big) ,\]
 where $G(n,m)$ denotes the Grassmannian of all $m$-dimensional linear subspaces ($m$-planes) of $\R^n$.\\
 For $\delta >0$, the set $\Sigma$ is called \textit{$\delta$-Reifenberg-flat of dimension $m$} if for all compact sets $K \subset \Sigma$ there exists a radius $r_K>0$
 such that 
  \[ \theta_K(r) := \sup_{x \in \Sigma \cap K} \theta_\Sigma(x,r) \leq \delta \Foa r\in (0,r_K]. \]
 $\Sigma$ is called Reifenberg-flat of dimension $m$ \textit{with vanishing constant} if $\Sigma$ is $\delta$-Reifenberg-flat of dimension $m$ for all $\delta >0$.
\end{defin}

It is easy to see that $\delta$-Reifenberg-flat sets do not have to be $C^1$-submanifolds. For example, for each fixed $\delta>0$, a $\delta$-Reifenberg-flat set of dimension 1 can 
be constructed as the graph of $u \colon \R \to \R: \ x \mapsto \delta \vert x \vert$, which is not a $C^1$-submanifold of $\R^2$. Moreover, even Reifenberg-flatness with vanishing constant is still not 
enough to guarantee $C^1$-regularity. It can be shown that the graph of 
 \[ u \colon \R \to \R,\ x \mapsto \sum_{k=1}^\infty \frac {\cos(2^kx)}{2^k\sqrt k}\]
is a Reifenberg-flat set with vanishing constant (see \cite{Toro_1997}). Nevertheless, although $u$ is continuous, it is nowhere differentiable. Moreover, T. Toro stated that the graph is not rectifiable in the sense of geometric measure theory, and therefore not a $C^1$-submanifold. 
We will show in detail with an indirect argument that $\operatorname{graph}(u)$ cannot be represented as a graph of a $C^1$-function in a neighbourhood of %$\binom{0}{u(0)}$ 
$(0,u(0))$ in Appendix \ref{sec:example_notC1}.

\vspace{\baselineskip}
There are a couple of variations to the definition of Reifenberg-flat sets with additional conditions, which guarantee more regularity than Reifenberg's Topological-Disk Theorem. 
If for a Reifenberg-flat set with vanishing constant there exists in addition, an exponent $\sigma \in (0,1]$ and for each compact set $K \subset \Sigma$ a constant $C_K >0$, such that the decay of the so-called $\beta$-numbers introduced by P. Jones in \cite{jones_1991} can be estimated as
\begin{align}   \beta_\Sigma(x,r) :=\frac 1 r \inf_{L \in G(n,m)} \left( \sup_{y \in \Sigma \cap B_r(x)} \dist(y,x+L) \right) \leq C_Kr^\sigma \Foa x \in K \AND r \leq 1,  \end{align}
then G. David, C. Kenig and T. Toro could show in \cite[Prop. 9.1]{david-kenig-toro_2001}, that $\Sigma$ is an embedded, $m$-dimensional $C^{1,\sigma}$-submanifold of $\R^n$.\\
A weaker assumption on $\Sigma \subset \R^n$ was stated by T. Toro in \cite{toro_1995} calling it \textit{$(\delta,\varepsilon, R)$-Reifenberg-flat at $x\in \Sigma$} for $\delta,\varepsilon,R>0$, if and only if
 \[ \theta_{B_R(x)}(r) \leq \delta \Foa r \in (0,R] \]
and 
 \begin{align} \int \limits_0^R \frac { \theta_{B_R(x)}(r)^2 } r \ dr \leq \varepsilon^2 .  \end{align}
In this setting it can be shown that there exist universal positive constants $\delta_0(m,n)$ and $\varepsilon_0(m,n)$, depending only on the dimensions $m$ and $n$, 
such that all sets $\Sigma \subset \R^n$ that are $(\delta,\varepsilon,R)$-Reifenberg-flat at all of their points with $0<\delta<\delta_0$, $0<\varepsilon<\varepsilon_0$, 
can be locally parameterized, on a scale determined by $R$, by bi-Lipschitz-homeomorphisms over open subsets of $\R^m$. In particular, such sets $\Sigma$ are embedded $C^{0,1}$-submanifolds of $\R^n$.

\vspace{\baselineskip}
In search of a characterization of $C^1$-submanifolds one may consider slightly stronger variants of Toro's integral condition in (2), which on the other hand, need to be 
weaker than the power-decay (1) of the $\beta$-numbers. We will present such a characterization in our main result, Theorem \ref{thm:c1regular} below, but first state a corollary of that result that uses an integral condition stronger than (2). 
This statement was independently proven by A. Ranjbar-Motlagh in \cite{ranjbar-motlagh}.

\vspace{\baselineskip}
\begin{thm} \label{thm:intcondition}
 Let $\Sigma \in \R^n$ be closed. If for all $x \in \Sigma$ there exists a radius $R_x >0$ such that
  \[ \int \limits_0^{R_x} \frac {\theta_{ B_{R_x}(x)}(r)} r \ dr < \infty,\]
  then $\Sigma$ is an embedded, $m$-dimensional $C^1$-submanifold of $\R^n$.
\end{thm}

Note that the dimension $m$ is encoded in the definition of the $\theta$-numbers; see Definition \ref{def:Reifenberg-flat}. Moreover, $\Sigma$ is not explicitly claimed to be Reifenberg-flat in Theorem \ref{thm:intcondition}, but the finite integral will ensure that $\Sigma$ is Reifenberg-flat with vanishing constant. %For details and the proof see section \ref{sec:intcon_proof}. 
Nevertheless, Theorem \ref{thm:intcondition} does not yet yield a 
characterization for $C^1$-submanifolds, since there are graphs of $C^1$-functions leading to an infinite integral. For example, let $u \colon (-1/2,1/2) \to \R$ be defined by
 \[ u(x)= \left \vert \int \limits_0^x \left(- \frac 2 {\log(y^2)}\right) \ dy \right \vert \Foa x \in \left( - \frac 1 2, \frac 1 2\right),  \]
then $u$ is of class $C^1$ on $(-1/2,1/2)$ and can be extended to a function $\tilde u \in C^1(\R)$. But $\Sigma:= \operatorname{graph}(\tilde u)\subset \R^2$ does \textit{not} satisfy the integral condition in Theorem \ref{thm:intcondition} as shown in detail in Appendix \ref{sec:example_proof}. Moreover, for every fixed $\alpha,\beta>0$ minor modifications of $u$ lead to a $C^1$-submanifold with 
 \[ \int \limits_0^{R_x} \frac {\theta^\beta_{ B_{R_x}(x)}(r)} {r^\alpha} \ dr = \infty.\]
A characterization for $C^1$-submanifolds using the condition of Reifenberg-flatness needs to allow $\theta$-numbers and the scale $r$ to decay more independently. 
Roughly speaking, a closed $\Sigma\subset \R^n$ is a $C^1$-submanifold, if and only if there exists a sequence of radii tending to zero, with controlled decay, such that $\Sigma$ satisfies the estimate for Reifenberg-flatness at these scales and the planes 
approximating $\Sigma$ converge to a limit-plane. We call this condition $(RPC)$ and the precise definition is as follows.

\begin{defin}[\textbf{R}eifenberg-\textbf{P}lane-\textbf{C}onvergence] \label{def:RPC}
 For $1\leq m <n$, we say $\Sigma \subset \R^n$ satisfies \textit{the condition $(RPC)$ with dimension $m$} if the following holds:\\
 For all $x \in \Sigma$ there exist a radius $R_x>0$, a sequence $(r_{x,i})_{i \in \N} \subset (0,R_x]$ and a constant $C_x >1$ with
  \[ r_{x,i+1} < r_{x,i} \leq C_x r_{x,i+1} \Foa i \in \N \AND \lim_{i \to \infty} r_{x,i}=0. \]
 Furthermore, there exist two sequences $(\delta_{x,i})_{i \in \N}, (\varepsilon_{x,i})_{i \in \N} \subset (0,1]$, both converging to zero, such that for all $y \in \Sigma \cap B_{R_x}(x)$ there exist planes $P(y,r_{x,i}), P_y \in G(n,m)$ with
  \[ \dist_{\HM}\Big(\Sigma \cap B_{r_{x,i}}(y), \big(y+P(y, r_{x,i})\big) \cap B_{r_{x,i}}(y) \Big) \leq \delta_{x,i}r_{x,i}\]
 and
  \[ \sphericalangle\big(P(y,r_{x,i}),P_y\big) \leq \varepsilon_{x,i}. \]
\end{defin}

Notice that the Grassmannian $G(n,m)$ equipped with the angle-metric is compact (see Definition \ref{def:metric}), so that every sequence of $m$-planes contains a converging subsequence, but the relation between the approximating planes $P(y,r_{x,i})$ and the scale $r_{x,i}$ is crucial in Definition \ref{def:RPC}. 
Notice also that $(RPC)$ does not explicitly claim that the set is Reifenberg-flat, since the approximation of $\Sigma$ is postulated only for a specific sequence of radii. 
Nevertheless, we show that $(RPC)$ is actually equivalent to Reifenberg-flatness with vanishing constant and uniformly converging approximating planes.\\
Here is our main result.

\begin{thm} \label{thm:c1regular} 
For a closed $\Sigma \in \R^n$ is equivalent:
 \begin{enumerate}
  \item $\Sigma$ satisfies $(RPC)$ with dimension $m$
  \item $\Sigma$ is an embedded, $m$-dimensional $C^1$-submanifold of $\R^n$
  \item$\Sigma$ is Reifenberg-flat with vanishing constant, and for all compact subsets $K \subset \Sigma$ and all $x \in K$ there exists an $m$-plane $L_x \in G(n,m)$ 
  such that
   \[ \sup_{x \in K} \sphericalangle \big( L(x,r),L_x \big)\xrightarrow[r \to 0]{}0, \]
  for all $L(x,r)\in G(n,m)$ with
   \[\sup_{x \in K} \frac 1 r \dist_{\HM}\Big(\Sigma \cap B_r(x),\big(x+ L(x,r)\big) \cap B_r(x)\Big) \xrightarrow[r \to 0]{} 0 \]
 \end{enumerate}
\end{thm}

As one can expect intuitively, in this case $P_x$ from condition $(RPC)$ and $L_x$ will coincide with the tangent plane $T_x\Sigma$.

\vspace{\baselineskip}
In Section \ref{sec:projection} we will review some basic facts about the Grassmannian and about orthogonal projections onto linear as well as onto affine subspaces of $\R^n$. 
Section \ref{sec:regularity} is dedicated to the proof of the main theorem and finally, in Section \ref{sec:intcon_proof} we will prove that the condition of Theorem \ref{thm:intcondition} is sufficient to obtain an embedded $C^1$-submanifold. The detailed structure of the examples mentioned in the introduction is presented in the appendix as well as the proofs of two technical lemmata 
%%%%%%%%%%%%%%%%%%%%%%%%%%%%%%%%%%%%%%%%%%%%%%%%%%%%%%%%%%%%%%%%%%%%%%%%%%%%%%%%%%%%%%%%%%%%%%%%%%

\section{\bf{Projections and preparations}} \label{sec:projection}

The aim of this section is to introduce all needed definitions and properties for linear and affine spaces, as well as for the projections onto those planes. 

\begin{defin} \label{def:grassmannian}
 For $n,m \in \N$ with $m\leq n$, the \textit{Grassmannian $G(n,m)$} denotes the set of all $m$-dimensional linear subspaces of $\R^n$.
\end{defin}

\begin{defin} \label{def:projection}
 For $P \in G(n,m)$, the orthogonal projection of $\R^n$ onto $P$ is denoted by $\pi_P$. Further $\pi_P^\perp:= id_{\R^n} - \pi_P$ shall denote 
 the orthogonal projection onto the linear
 subspace perpendicular to $P$.
\end{defin}

Using orthogonal projections it is possible to define a distance between two elements of $G(n,m)$.

\begin{defin} \label{def:metric}
 For two planes $P_1,P_2 \in G(n,m)$ the \textit{included angle} is defined by
  \[ \sphericalangle(P_1,P_2):= \Vert \pi_{P_1}-\pi_{P_2} \Vert := \sup_{x \in \S^{n-1}} \vert \pi_{P_1}(x) -\pi_{P_2}(x)\vert. \]
 The angle $\sphericalangle(\cdot,\cdot)$ is a metric on the Grassmannian $G(n,m)$.
\end{defin}

Together with this metric, the Grassmannian $(G(n,m), \sphericalangle(\cdot,\cdot))$ is a compact manifold. The following lemma allows to use different useful presentations for 
the angle between two planes.

\begin{lem}[8.9.3 in \cite{allard_1972}] \label{lem:metric}
 Let $P_1,P_2 \in G(n,m)$, then
  \[ \Vert \pi_{P_1}-\pi_{P_2} \Vert = \Vert \pi_{P_1}^\perp - \pi_{P_2}^\perp \Vert
  =  \Vert \pi_{P_1}^\perp \circ \pi_{P_2} \Vert = \Vert \pi_{P_1} \circ \pi_{P_2}^\perp \Vert= \Vert \pi_{P_2}^\perp \circ \pi_{P_1}\Vert 
  =  \Vert \pi_{P_2} \circ \pi_{P_1}^\perp \Vert.\]
\end{lem}

Citing the first part of Lemma 2.2 in \cite{kolasinski-strzelecki-vdm_2015} we get 

\begin{lem} \label{lem:projectioninv}
 Assume $P_1,P_2 \in G(n,m)$. If $\sphericalangle(P_1,P_2)<1$, then the projection $\pi_{P_1\mid P_2} \colon P_2 \to P_1$ is a linear isomorphism.
\end{lem}

Although we use linear spaces most of the time, it is also necessary to define projections onto affine spaces and the angles between those.

\begin{defin} \label{def:projectionaffine}
 For $x \in \R^n$ and $P \in G(n,m)$, the orthogonal projection onto $Q:=x+P$ and the corresponding perpendicular plane are defined by
  \[\pi_Q(z) := x + \pi_P(z-x) \]
 and
  \[\pi_Q^\perp(z) = z- \pi_Q(z) =(z-x)-\pi_P(z-x) = \pi_P^\perp(z-x).\]
 Moreover, for $x_1,x_2 \in \R^n$ and $P_1,P_2 \in G(n,m)$ the angle between $Q_1:=x_1+P_1$ and $Q_2:=x_2 + P_2$ is defined as
  \[ \sphericalangle(Q_1,Q_2) := \sphericalangle(P_1,P_2). \]
\end{defin}

For a smooth function's graph, \cite[8.9.5]{allard_1972} leads to an estimate for the angle between tangent spaces.

\begin{lem} \label{lem:metricinequ}
 Let $\alpha \geq 0$, $P \in G(n,m)$ and assume $f \in C^1(P,P^\perp)$ satisfies $\Vert f' \Vert \leq \alpha$ and $f'(0)=0$. 
 Let $g(x) := x+f(x)$ and $\Sigma := g(P)$ be the graph of $f$, then for all $x,y \in P$ the following estimates hold:
  \[ \Vert \pi_{T_{g(y)}\Sigma} - \pi_{T_{g(x)}\Sigma} \Vert \leq \Vert f'(x) - f'(y) \Vert \leq \sqrt{\frac {1+\alpha^2}{1-\alpha^2} } \Vert \pi_{T_{g(y)}\Sigma} - \pi_{T_{g(x)}\Sigma} \Vert\]
\end{lem}

Lastly there is an estimate for angles between planes, in a more generel setting.

\begin{lem}[Prop. 2.5 in \cite{kolasinski-strzelecki-vdm_2012}]\label{lem:winkelabsch}
 Let $P_1,P_2 \in G(n,m)$ and let $(e_1,\dots,e_m)$ be some orthonormal basis of $P_1$. Assume that for each $i=1,\dots,m$ we have the estimate $\dist(e_i,U)\leq \theta$ 
 for some $\theta \in (0,1/\sqrt 2).$ Then there exists a constant $C_1=C_1(m)$ such that
  \[ \sphericalangle(P_1,P_2)\leq C_1\theta. \]
\end{lem}

%%%%%%%%%%%%%%%%%%%%%%%%%%%%%%%%%%%%%%%%%%%%%%%%%%%%%%%%%%%%%%%%%%%%%%%%%%%%%%%%%%%%%%%%%%%%%%%%%%

\section{\bf{Equivalence of (RPC) and $C^1$-regularity}} \label{sec:regularity}

In this section we prove the main theorem. First we will show that $(RPC)$ is equivalent to Reifenberg-flatness with vanishing constant and a uniform convergence of approximating planes. This allows us to use $(RPC)$ and Reifenberg-flatness to prove that every set, which satisfies $(RPC)$ is an embedded $C^1$-submanifold. 
We will approach this by using a different characterization, namely writing $\Sigma$ locally as the graph of a $C^1$-function. It turns out ,that for an element $x \in \Sigma$ the radius $r$ providing $\Sigma \cap B_r(x)$ can be represented as a graph, can be given depending on the ratio of decay of $\delta_{x,i},\varepsilon_{x,i}$ and $r_{x,i}$.\\
Lastly we will show the other implication, using that the representation as a graph of a smooth function already provides Reifenberg-flatness.

Notice that we will fix the dimension $m$ of a subset $\Sigma \subset \R^n$ and say that $\Sigma$ is a $\delta$-Reifenberg-flat set or satisfies $(RPC)$ without mentioning the dimension.

\begin{lem} \label{lem:abstandebene}
 Assume $\Sigma \subset \R^n$ satisfies (RPC), then for all $x \in \Sigma$ we get
  \[ \dist (z,y+ P_y) \leq w_x(\vert z-y \vert) \cdot \vert z-y \vert \Foa y \in \Sigma \cap B_{R_x}(x) \AND  {z \in \Sigma \cap B_{r_{x,1}}(y)}, \]
 where the function $w_x \colon \R \to \R$ is given by
  \[ w_x(r)= \varepsilon_{x,i} + C_x \delta_{x,i} \Foa r \in (r_{x,i+1},r_{x,i}]. \]
\end{lem}

Note that $w_x$ is a piecewise constant function with $\lim_{r \to 0} w_x(r)=0$. It is possible for $w_x$ to be not monotonically decreasing, because 
$(RPC)$ require this neither for $\delta_{x,i}$ nor for $\varepsilon_{x,i}$.

\begin{proof}
 Let $x \in \Sigma$ and $y \in \Sigma \cap B_{R_x}(x)$ be fixed. For $z \in \Sigma \cap B_{r_{x,1}}(y)$ there exists an $i \in \N$ with
 $\vert z-y\vert \in (r_{x,i+1},r_{x,i}]$. This yields
 \begin{align*}
  \dist(z,y+P_y) &= \vert \pi_{P_y}^\perp(z-y)\vert \\
  &\leq \left\vert \left(  \pi_{P_y}^\perp-\pi_{P(y,r_{x,i})}^\perp\right) (z-y) \right\vert + \vert \pi_{P(y,r_{x,i})}^\perp (z-y) \vert\\
  &\leq \varepsilon_{x,i} \vert z-y \vert + \delta_{x,i} r_{x,i}\\
  &\leq \varepsilon_{x,i} \vert z-y \vert + \delta_{x,i}C_x \vert z-y \vert.
 \end{align*}
\end{proof}

The idea of Lemma \ref{lem:winkelabsch} will frequently be used for Reifenberg-flat sets $\Sigma$ 
while $P_1$ and $P_2$ are the approximating planes of Definition \ref{def:Reifenberg-flat} for either different or the same radii and points of $\Sigma$. The following lemma uses Lemma \ref{lem:winkelabsch} to get an estimate in this
setting.

\begin{lem} \label{lem:reifenbergwinkel}
 Let $x_1,x_2 \in \Sigma \subset \R^n$, $0<r_1 \leq r_2$, $ \delta_1, \delta_2 \in (0, \frac 1 2)$ and $P_1,P_2 \in G(n,m)$ be given such that
  \[ \vert x_1 - x_2 \vert <\frac {r_1} 2 \]
 and 
  \[ \dist_{\HM}\Big( \Sigma \cap B_{r_j}(x_j), (x_j+P_j) \cap B_{r_j}(x_j)\Big) \leq \delta_jr_j \Fo j=1,2. \]
 If 
  \[ \frac 2 {1-2\delta_1} \left( \delta_1 + 2\frac {r_2}{r_1} \delta_2\right)< \frac 1 {\sqrt 2}, \]
 then we get
  \[ \sphericalangle(P_1,P_2) \leq C_1 \frac 2 {1-2\delta_1} \left( \delta_1 + 2\frac {r_2}{r_1} \delta_2\right) . \]
\end{lem}

\begin{proof}
 Let $(e_1,\dots, e_m)$ be an orthonormal basis of $P_1$. Define 
 \begin{align*}
  y_0&:=x_1
  \intertext{and}
  y_i&:=x_1+\frac {1-2\delta_1} 2 r_1 e_i \Fo i=1, \dots,m.
 \end{align*}
 For all $i =1,\dots, m$ there exists a $z_i \in \Sigma \cap B_{r_1}(x_1)$ with
  \[ \vert z_i - y_i \vert \leq r_1 \delta_1. \]
 Note that for $z_0:=y_0=x_0$, the point $z_0$ is also an element of $\Sigma \cap B_{r_1}(x_1) \cap B_{r_2}(x_2)$. Further we get
  \[ \vert z_i -x_1 \vert \leq \vert z_i - y_i \vert + \vert y_i -x_1 \vert \leq r_1 \delta_1 + r_1 \frac {1-2\delta_1} 2= \frac {r_1}2 \Foa i = 1 , \dots, m.  \]
 This leads to
 \begin{align*}
  \vert z_i - x_2 \vert &\leq \vert z_i-x_1 \vert + \vert x_1 -x_2\vert \\
  &< r_1\left( \frac 1 2 + \frac 1 2 \right) \\
  &= r_1 \leq r_2 \Foa i =1, \dots ,m.
  \end{align*}
 Therefore for every $i=0,\dots,m$ there exists a $w_i \in (x_2 +P_2) \cap B_{r_2}(x_2)$ with
  \[ \vert w_i - z_i \vert \leq r_2 \delta_2. \]
 Define $\tilde y_i := y_i-y_0$ and $\tilde w_i := w_i-w_0$ for $i=1 , \dots , m.$ Then $\tilde y_i/\vert\tilde y_i \vert = e_i$ is obviously an orthonormal 
 basis of $P_1$ and $\tilde w_i/\vert \tilde y_i \vert$ is an element of $P_2$. The previous estimates yield
 \begin{align*}
  \left \vert \frac {\tilde y_i} {\vert \tilde y_i \vert } - \frac {\tilde w_i} {\vert \tilde y_i \vert } \right \vert &= \frac 1 {\vert \tilde y_i \vert} 
   \bigg \vert y_i -y_0 -w_i + w_0\bigg \vert \\
  &=\frac 2 {(1-2\delta_1 )r_1} \bigg \vert y_i - z_i + z_0 -y_0 + z_i -w_i + w_0 - z_0 \bigg \vert\\
  &\leq \frac 2 {(1-2\delta_1 )r_1}  (r_1 \delta_1 + 0 + r_2 \delta_2 + r_2 \delta_2)\\
  & \leq \frac 2 {1-2\delta_1 } \left( \delta_1 +2 \frac {r_2}{r_1} \delta_2 \right) \Foa i=1,\dots,m.
 \end{align*}
 This is assumed to be strictly less than $1/\sqrt2$ and therefore Lemma \ref{lem:winkelabsch} leads to
  \[ \sphericalangle(P_1,P_2) \leq C_1(m) \frac 2 {1-2\delta_1 } \left( \delta_1 +2 \frac {r_2}{r_1} \delta_2 \right).\]
\end{proof}

Now we will show that every set satisfying $(RPC)$ is indeed Reifenberg-flat with vanishing constant. Moreover, we will see that $(RPC)$ is an even 
stronger assumption and allows to approximate the set for a fixed point with the same plane at each scale. In fact, we will show the estimation for Reifenberg-flatness only for 
a ball around $x \in \Sigma$. By a covering argument, we later see, that the estimate holds true for all compact subsets of $\Sigma$.

\begin{lem} \label{lem:reifenbergPy}
 Assume $\Sigma \subset \R^n$ satisfies $(RPC)$, then for all $x \in \Sigma$ and $k \geq \tilde k_x$, where $\tilde k_x \in \N$ denotes the index with 
  \[ \delta_{x,k} < \frac 1 {C_x} \Foa k \geq \tilde k_x ,\]
 we get
 \begin{align*}
   \sup \limits_{y \in B_{R_x}(x) \cap \Sigma} \frac 1 r \dist_{\HM}\Big( \Sigma \cap B_r(y),(y+ P_y) \cap B_r(y) \Big) 
  &\leq \sup \limits_{i \geq k} (\varepsilon_{x,i} +2 C_x \delta_{x,i})\\
  &=: \tilde \delta_{x,r} \Foa r \leq r_{x,k}.
 \end{align*}
\end{lem}

Note that the existence of $\tilde k_x$ is an immidiate result of $\delta_{x,k}$ tending to zero. The value of $\tilde k_x$ and therefore the scale of the approximation depends highly on the point $x \in\Sigma$.

\begin{proof}
 Let $x \in \Sigma$ be fixed, $y \in \Sigma \cap B_{R_x}(x)$ and $z \in \Sigma \cap B_{r}(y)$ for a radius $r \in (0,r_{x,\tilde k_x}]$. Then for $y \neq z$ there exists an 
 $i \in \N$ with $r_{x,i+1} < \vert z-y \vert \leq r_{x,i}$ and Lemma \ref{lem:abstandebene} leads to 
 \begin{align*}
  \frac 1 r \dist\Big(z, (y+P_y) \cap B_{r}(y)\Big) &\leq \frac 1 r w_x(\vert z-y \vert) \cdot \vert z-y \vert\\
  &\leq w_x(\vert z-y \vert) \\
  &= \varepsilon_{x,i} + C_x \delta_{x,i}.
 \end{align*}
 Let $k \in \N$ such that $r_{x,k+1}<r \leq r_{x,k}$, then this implies
  \[ \sup \limits_{z \in \Sigma \cap B_r(y)} \frac 1 r \dist\Big(z, (y+P_y) \cap B_r(y)\Big) \leq \sup_{i\geq k } (\varepsilon_{x,i} + C_x \delta_{x,i}). \]
 Moreover, we have $k\geq \tilde k_x$. Using the definition of $\tilde k_x$ we have
  \[r- r_{x,k} \delta_{x,k} \geq r - r C_x \delta_{x,r} >0 .\]
 For $z \in (y+ P_y) \cap B_{r-r_{x,k}\delta_{x,k}}(y)$ defining
  \[ \tilde z:=y + \pi_{P(y,r_{x,k})} (z-y), \]
 leads to
  \[ \vert \tilde z -y \vert = \vert \pi_{P(y,r_{x,k})}(z-y ) \vert \leq \vert z-y \vert < r - r_{x,k}\delta_{x,k} <r \leq r_{x,k}. \]
 Hence there exists a $w \in \Sigma \cap B_{r_{x,k}}(y)$ with
  \[ \vert \tilde z -w \vert \leq r_{x,k} \delta_{x,k}. \]
 Moreover
  \[ \vert w -y \vert \leq \vert w- \tilde z \vert + \vert \tilde z -y \vert < r_{x,k} \delta_{x,k} + r - r_{x,k}\delta_{x,k} =r  \]
 and therefore $w \in \Sigma \cap B_r(y)$. Using $z-y \in P_y$ and Lemma \ref{lem:metric}, we get
 \begin{align*}
  \dist\Big(z, \Sigma \cap B_{r}(y)\Big) &\leq \vert z - w \vert\\
  &\leq \vert z- \tilde z \vert + \vert \tilde z - w\vert \\
  &= \vert \pi_{P(y,r_{x,k})}^\perp (z-y) \vert + \vert \tilde z -w \vert\\
  &\leq \varepsilon_{x,k} \vert z-y\vert + r_{x,k } \delta_{x,k}\\
  &\leq r\left( \varepsilon_{x,k} + C_x \delta_{x,k} \right).
 \end{align*}
 Now let $z \in (y+ P_y) \cap ( B_r(y) \setminus B_{r-r_{x,k}\delta_{x,k}}(y))$, then there exists a $z'\in (y+P_y) \cap B_{r-r_{x,k}\delta_{x,k}}(y)$ such that
  \[ \vert z' -z\vert < r_{x,k} \delta_{x,k}. \]
 Therefore we get a $w \in \Sigma \cap B_r(y)$ with 
 \begin{align*}
  \vert w - z \vert &\leq \vert w - z' \vert + \vert z' -z \vert \\
   &\leq r\left( \varepsilon_{x,k} + C_x \delta_{x,k} \right) + r_{x,k} \delta_{x,k}\\
   &\leq r\left( \varepsilon_{x,k} + 2 C_x \delta_{x,k} \right).
 \end{align*}
 Finally
 \begin{align*}
  \frac 1 r \dist_{\HM}\Big( \Sigma \cap B_{r}(y) , (y+P_y) \cap B_r(y) \Big) &\leq \max \left\{\sup_{i\geq k} (\varepsilon_{x,i} + C_x \delta_{x,i}), \varepsilon_{x,k} + 2 C_x \delta_{x,k} \right\} \\
  &\leq \sup_{i\geq k } (\varepsilon_{x,i} + 2 C_x \delta_{x,i}),
 \end{align*}
 which is independent of $y \in B_{R_x}(x)$ and implies the postulated statement.
\end{proof}

\begin{rem} \label{rem:tildedelta}
 Note that $\tilde \delta_{x,k}$ is monotonically decreasing and using the convergence of $\delta_{x,i}$ and $\varepsilon_{x,i}$ we get $\tilde \delta_{x,k} \to 0$ as $k \to \infty$. 
 Lemma \ref{lem:reifenbergPy} then implies that $\Sigma$ is a $\delta$-Reifenberg-flat set for all $\delta>0$, i.e. it is Reifenberg-flat with vanishing constant. 
 Moreover, the plane which approximates $\Sigma$ at the point $y \in \Sigma$ with respect to the $\delta$-Reifenberg-flatness can be fixed as $y + P_y$ for all small radii.
\end{rem}

For a set $\Sigma \subset \R^n$ which satisfies $(RPC)$ and $y \in \Sigma$ the plane $P_y$ arises as a limit of planes $P(y,r_{x,i})$. Up to this point, we did 
not mention that these planes might also depend on $x$ and that we should have writen $P_y^x$, but in fact, we are now ready to show, that the $P_y^x$ are the same 
for all $x \in \Sigma$ with $y \in \Sigma \cap B_{R_x}(x)$. Moreover, we get an estimate for the angle between two planes $P_y$ and $P_z$, whenever $z$ is an 
element of $\Sigma \cap B_{R_x}(x)$ with $\vert y-z\vert$ small enough.

\begin{lem} \label{lem:pyx=pyz}
 Assume $\Sigma \subset \R^n$ satisfies $(RPC)$. 
 \begin{enumerate}
  \item  For $x,\tilde x \in \Sigma$ we get
   \[ P_y^x = P_y^{\tilde x} \Foa y \in \Sigma \cap B_{R_x}(x) \cap B_{R_{\tilde x}}(\tilde x). \]
  \item For $x \in \Sigma$, $k \geq \tilde k_x$ and $y,z \in \Sigma \cap B_{R_x}(x)$ with 
  $\vert z-y \vert < \frac {r_{x,k}} 2$ and $\tilde \delta_{x,k} <\frac 1 {11}$ we get 
   \[ \sphericalangle(P_y,P_z) \leq \frac {22} 3 C_1(m) \tilde \delta_{x,k}=:C_2(m)\tilde \delta_{x,k}. \]
 \end{enumerate}
\end{lem}

\begin{proof}
 \begin{enumerate}[wide, labelwidth=!, labelindent=0pt]
  \item Let $x, \tilde x \in \Sigma$ and $y \in \Sigma \cap B_{R_x}(x) \cap B_{R_{\tilde x}}(\tilde x)$. The sequences $\varepsilon_{x,k}$ and $\varepsilon_{\tilde x , k}$ 
  converge to zero and hence for all $\varepsilon >0$ there exist an $N_1 \in \N$ such that 
   \[ \varepsilon_{x,k},\varepsilon_{\tilde x,k} \leq \frac \varepsilon 3 \Foa k \geq N_1. \]
  Moreover, there exists an $N_2 \in \N$ with $N_2>N_1$ and
   \[ \delta_{x,k} < \min\left\{ \frac \varepsilon {24 C_1}, \frac 1 4 \right\} \AND \delta_{\tilde x, k} < \frac \varepsilon {48 C_1 C_x} \Foa k \geq N_2.  \]
  Define 
  \begin{align*}
   k&:=\begin{cases}
       N_2 &\Fo r_{\tilde x,N_2} \leq r_{x,N_2}, \\
       \min \{ l \in \N \mid \ r_{\tilde x, l } \leq r_{x,N_2} \} &\Fo r_{\tilde x,N_2} > r_{x,N_2},
      \end{cases}\\
   \intertext{and}
   i&:=\min\{l \in \N \mid \ r_{x,l} \leq r_{\tilde x , k}\}.&
  \end{align*}
  Then we have $k,i \geq N_2$ and
   \[ r_{x,i} \leq r_{\tilde x , k } \leq r_{x,i-1}. \]
  Let $\varepsilon$ be sufficiently small, i.e. $\frac \varepsilon {3C_1} < \frac 1 {\sqrt 2}.$ Then
  \begin{align*}
   \frac 2 {1-2\delta_{x,i}} \left( \delta_{x,i} + \frac {r_{\tilde x,k}}{r_{x,i}} \delta_{\tilde x,k} \right) &\leq  4 (\delta_{x,i} + 2 C_x \delta_{\tilde x,k})\\
   &\leq 4 \left( \frac \varepsilon { 24C_1 } + 2 C_x \frac \varepsilon {48C_1C_x} \right) \\
   &= \frac \varepsilon {3C_1} \\
   &< \frac 1 {\sqrt 2}.
  \end{align*}
  Using Lemma \ref{lem:reifenbergwinkel} we get 
  \begin{align*}
   \sphericalangle \big( P(y,r_{x,i}), P(y,r_{\tilde x,k}) \big) &\leq C_1 \frac 2 {1-2\delta_{x,i}} \left( \delta_{x,i} + 2 \frac {r_{\tilde x,k}}{r_{x,i}} \delta_{\tilde x ,k} \right) \\
    &\leq \frac \varepsilon 3.
  \end{align*}
  Finally
  \begin{align*}
   \sphericalangle\left( P_y^x, P_y^{\tilde x} \right)& \leq \sphericalangle\big(P_y^x, P(y,r_{x,i})\big) + \sphericalangle\big(P(y,r_{x,i}),P(y,r_{\tilde x,k})\big) + 
   \sphericalangle\big(P(y,r_{\tilde x , k}),P_y^{\tilde x}\big) \\
   & \leq \varepsilon.
  \end{align*}
  The limit $\varepsilon \to 0$ implies
   \[ P_y^x = P_y^{\tilde x}.\]
  \item For $y,z \in \Sigma \cap B_{R_x}(x)$, $k\geq \tilde k_x$ and $r\leq r_{x,k}$ Lemma \ref{lem:reifenbergPy} leads to
  \begin{align*}
   \dist_{\HM}\Big( \Sigma \cap B_{r}(y),(y+P_y) \cap B_r(y)\Big) \leq r \tilde \delta_{x,k}\\
   \intertext{and}\\
   \dist_{\HM}\Big( \Sigma \cap B_{r}(z),(z+P_z) \cap B_r(z)\Big) \leq r \tilde \delta_{x,k}.\\
  \end{align*}
  If $\vert z-y \vert < \frac {r_{x,k}} 2$ and $\tilde \delta_{x,k}< \frac 1 {11}$, then
   \[ \frac {2}{1-2\tilde \delta_{x,k}} (\tilde \delta_{x,k} + 2\tilde \delta_{x,k}) < \frac {22}{ 3} \tilde \delta_{x,k} < \frac 1{\sqrt 2} \]
  and for $r_1:=r_2:=r_{x,k}$ and $\delta_1:=\delta_2:=\tilde \delta_{x,k}$ Lemma \ref{lem:reifenbergwinkel} yields
   \[ \sphericalangle(P_y,P_z) \leq \frac {22} 3 C_1(m) \tilde \delta_{x,k},\]
  which completes the proof.
 \end{enumerate}
\end{proof}

\vspace{3\baselineskip}
%\newpage
\begin{lem} \label{lem:Reifenbergaequ}
 For closed $\Sigma \subset \R^n$, the following statements are equivalent:
 \begin{enumerate}
  \item $\Sigma$ satisfies $(RPC)$
  \item $\Sigma$ is Reifenberg-flat with vanishing constant and, for all compact subsets $K \subset \Sigma$ and all $x \in K$ there exists a plane $L_x \in G(n,m)$ 
  such that
   \[ \sup_{x \in K} \sphericalangle \big( L(x,r),L_x \big)\xrightarrow[r \to 0]{}0, \]
  for all $L(x,r)\in G(n,m)$ with
   \[\sup_{x \in K} \frac 1 r \dist_{\HM}\Big(\Sigma \cap B_r(x),\big(x+ L(x,r)\big) \cap B_r(x)\Big) \xrightarrow[r \to 0]{} 0 \]
 \end{enumerate}
\end{lem}

Note that the existence of planes $L(x,r)$, which approximate $\Sigma$ with respect to the Reifenberg-flatness such that their distances to $\Sigma$ converges uniformly to zero is 
already guaranteed by the Reifenberg-flatness with vanishing constant. Only the existence of a limit-plane is an additional condition to the Reifenberg-flatness in \ref{lem:Reifenbergaequ} (2). 
Obviously, $L_x$ and $P_x$ will coincide.

\begin{proof}
 $"(1) \Rightarrow (2)":$ For fixed $x \in \Sigma$ using Lemma \ref{lem:reifenbergPy} yields for $k \geq \tilde k_x$
  \[ \sup_{y \in \Sigma \cap B_{R_x}(x)} \frac 1 r \dist_{\HM} \Big(\Sigma \cap B_r(y),(y+ P_y) \cap B_r(y) \Big) \leq \tilde \delta_{x,k} \Foa r \leq r_{x,k}.\]
 For a compact set $K \subset \Sigma$ we have
  \[ K \subset \bigcup_{x \in K} B_{R_x}(x) \]
  and the compactness provides $x_1,\dots ,x_N \in K$ with
  \[ K \subset \bigcup \limits_{i=1}^N B_{R_{x_i}}(x_i). \]
 Let $\tilde k \in \N$ be defined by $\tilde k := \max\{\tilde k_{x_1},\dots , \tilde k_{x_N}\}$. 
 For given $\delta>0$ and $i \in \{1,\dots,N\}$ the convergence of $\tilde \delta_{x_i,k}$ to zero guarantees that there is a $j(x_i,\delta)\geq \tilde k$ such that
 $\tilde \delta_{x_i,j(x_i,\delta)} \leq \delta$. This implies
  \[ \sup_{y \in \Sigma \cap B_{R_{x_i}}(x_i)} \frac 1 r \dist_{\HM} \Big(\Sigma \cap B_r(y),(y+ P_y) \cap B_r(y) \Big) \leq \tilde \delta_{x_i,j(x,\delta)} \leq \delta \Foa r \leq r_{x_i,j(x_i,\delta)} .\]
 Now define $r_0=r_0(\delta):= \min\{r_{x_1,j(x_1,\delta)}, \dots ,r_{x_N,j(x_N,\delta)} \}$. Then we get
 \begin{align*}
  &\phantom{\leq}\sup_{y \in K} \frac 1 r \dist_{\HM} \Big(\Sigma \cap B_r(y),(y+ P_y) \cap B_r(y) \Big) \\
  &\leq \max_{i=1,\dots,N} \sup_{y \in \Sigma \cap B_{R_{x_i}}(x)} \frac 1 r \dist_{\HM} \Big(\Sigma \cap B_r(y), (y+P_y) \cap B_r(y) \Big)\\
  &\leq \delta \Foa r \leq r_0.
 \end{align*}
 This holds true for every arbitrary $\delta>0$ implying that $\Sigma$ is a Reifenberg-flat 
 set with vanishing constant and fixed approximating plane.\\
 Now let $x\in K$ and $L(x,r) \in G(n,m)$ be a plane, depending on $x$ and $r$, such that
  \[\frac 1 r \dist_{\HM}\Big(\Sigma \cap B_r(x),\big(x+L(x,r)\big)\cap B_r(x)\Big)=:\delta(x,r) \xrightarrow[r\to0]{}0 .\]
 We have to show that $L(x,r)$ converges to a limit plane $L_x \in G(n,m)$ and in fact we will show $L_x =P_x$.\\
 For $x_1=x_2=x$, $r_1=r_2=r$, $P_1=L(x,r)$, $P_2=P_y$, $\delta_1=\delta(x,r)$ and $\delta_2=\tilde \delta_{x,k(r)}$, where $k(r)$ is defined such that $r_{x,k(r)+1}<r \leq r_{x,k(r)}$, 
 we have $\delta_1,\delta_2< \frac 1 2$ for $r$ small enough, as well as
  \[ \frac 2 {1-2\delta(x,r)} \left(\delta(x,r) + 2 \tilde \delta_{x,k(r)}\right) < \frac 1 {\sqrt 2},\] 
 Lemma \ref{lem:reifenbergwinkel} leads to
 \begin{align*}
  \lim_{r \to 0}\sphericalangle\big(L(x,r),P_y\big) & \leq \lim_{r \to 0}C_1(m) \frac 2 {1-2\tilde \delta_{x,k(r)}} \left( \delta(r) + 2 \tilde \delta_{x,k(r)}\right) = 0.
 \end{align*}
 $"(2) \Rightarrow (1)":$ For $x \in \Sigma$ define $R_x:=1$, $C_x>1$ arbitrary and a sequence $r_{x,i} \subset (0,1]$ with 
 $r_{x,i+1} \leq r_{x,i} \leq C_x r_{x,i+1}$ and $r_{x,i} \xrightarrow[i \to \infty]{}0$.\\
 The compactness of $(G(n,m),\sphericalangle(\cdot,\cdot))$ implies that for $y \in \Sigma \cap B_{R_x}(x)$ there exists a minimizer of
  \[ L \mapsto \frac 1 {r_{x,k}} \dist_{\HM}\Big( \Sigma \cap B_{r_{x,k}}(y), (y+L) \cap B_{r_{x,k}}(y) \Big). \]
 Let $P(y,r_{x,k})$ denote this minimizer. Define
  \[ \delta_{x,k} := \sup_{y \in \Sigma \cap \overline{B_{R_x}(x)}} \frac 1 {r_{x,k}} \dist_{\HM}\Big( \Sigma \cap B_{r_{x,k}}(y), \big(y+ P(y,r_{x,k})\big) \cap B_{r_{x,k}}(y)  \Big) .\] 
 The Reifenberg-flatness with vanishing constant guarantees $\delta_{x,k} \xrightarrow[k \to \infty]{} 0$.
 Finally, themade assumptions imply that for all $y \in \Sigma \cap B_{R_x}(x)$ there exists a $P_y:=L_y \in G(n,m)$ with
  \[ \sup_{y \in \Sigma \cap B_{R_x}(x)} \sphericalangle\big(P(y,r_{x,k}),P_y\big) =: \varepsilon_{x,k} \xrightarrow[k \to \infty]{}0. \]
\end{proof}

$\Sigma$ being a $C^1$-submanifold, is equivalent to $\Sigma$ locally being a graph of a $C^1$-function. Therefore it is a necessary condition, that for each $x \in \Sigma$ 
there exists a plane $P \in G(n,m)$ such that the orthogonal projection $\pi_{x+P\mid \Sigma}$ is locally bijective onto an open subset of $x +P$. Both, the injectivity and surjectivity 
will be results of the Reifenberg-flatness of $\Sigma$. $(RPC)$ guarantees for $\Sigma$ to be Reifenberg-flat with vanishing constant, which allows us to use Lemma \ref{lem:projectionsurj}, stated for 
codimension 1 in \cite{david-kenig-toro_2001} and ensuring the surjectivity. Although the main argument of \cite{david-kenig-toro_2001} does not depend on the dimension, we will present the proof of Lemma \ref{lem:projectionsurj} and \ref{lem:reifenbergfunction}, which is also part of \cite{david-kenig-toro_2001}, in appendix \ref{sec:reifenbergproof} to make sure, 
that this result still holds for higher codimension.\\
Lemma \ref{lem:reifenbergfunction} yields a parameterization for Reifenberg-flat sets, which is often used to achieve more results for Reifenberg-flat sets. Here we will need this parameterization only to prove Lemma \ref{lem:projectionsurj}.

\begin{lem} \label{lem:reifenbergfunction}
 There exists a $\delta_0>0$ such that for every closed, $m$-dimensional $\delta$-Reifenberg-flat set $\Sigma \subset \R^n$ with $\delta\leq \delta_0$ and $x \in \Sigma$ there 
 is a $R_0=R_0(x,\delta, \Sigma)>0$ such that for all $L \in G(n,m)$ with
  \[ \dist_{\HM}\Big( \Sigma \cap B_r(x), (x+L) \cap B_r(x) \Big) \leq r\delta \Fo r \leq R_0 \]
 exists a continuous function 
  \[ \tau \colon (x+L) \cap \overline{B_{\frac {15}{16}r}(x)} \to \Sigma \cap \overline{B_r(x)} \]
 with
  \[ \vert \tau(y)-y \vert \leq Cr\delta \leq \frac 5 {144} r \Foa y \in (x+L) \cap \overline{B_r(x)}. \]
\end{lem}

The constants $\delta_0$ and $R_0$ can be set as $\delta_0< (48(3C_1(m)+2))^{-1}$ and $R_0(x,\delta,\Sigma)>0$ small enough, such that
 \[ \frac 1 r \inf_{L\in G(n,m)} \dist_{\HM} \Big( \Sigma \cap B_r(y) , (y+L) \cap B_r(y)  \Big) \leq \delta \Foa y \in \Sigma \cap \overline{B_{R_0}(x)}. \]
Such an $R_0(x,\delta,\Sigma)$ exists, because of the Reifenberg-flatness.

\begin{lem} \label{lem:projectionsurj}
 For all closed, $\delta$-Reifenberg-flat sets $\Sigma \subset \R^n$ with $\delta \leq \delta_0$, 
 all $x \in \Sigma$ and $L \in G(n,m)$ with
  \[ \frac 1 r \dist_{\HM}\Big( \Sigma \cap B_r(x), (x+L) \cap B_r(x) \Big) \leq \delta \Fo r\leq R_0, \]
 we get 
  \[ (x+L) \cap B_{\frac r 4}(x) \subset \pi_{x+L} \left( \Sigma \cap B_{\frac r 2}(x) \right), \]
 where $\delta_0$ and $R_0$ are as stated in Lemma \ref{lem:reifenbergfunction}.
\end{lem}

We are now ready to prove Theorem \ref{thm:c1regular} in two steps. First we will see that if $\Sigma$ satisfies $(RPC)$, it is locally a graph of a $C^1$ function, i.e. it is an embedded $C^1$-submanifold. 
Finally we prove that every embedded $C^1$-submanifold satisfies the $(RPC)$ condition.

\begin{lem} \label{lem:RPCisC1}
 Assume $\Sigma \subset \R^n$ is closed and satisfies $(RPC)$ with dimension $m$, then for all $x \in \Sigma$ there exist a radius $r_x$ and a function 
 $u_x \in C^1(P_x,P_x^\perp)$ with
  \[ \left( \Sigma \cap B_{r_x}(x)\right) -x = \operatorname{graph}(u_x) \cap B_{r_x}(0), \]
 i.e. $\Sigma$ is an embedded, $m$-dimensional $C^1$-submanifold of $\R^n$.
\end{lem}

Note that the radius $r_x$ can be given explicitly by $\frac 1 3 r_{x,k}$ for $k\in \N_{>1}$ such that $\tilde \delta_{x,k-1}< \min\{(48(3C_1(m)+2))^{-1},(6C_2(m)+2C_x)^{-1}\}$. 
Therefore, the radius for the neighbourhood, where $\Sigma$ can be represented as a $C^1$-graph 
depends only on the dimension of $\Sigma$ and the ratio of decay between the sequences $\delta_{x,i},\varepsilon_{x,i}$ and $r_{x,i}$.

\begin{proof}
 Let $x$ be fixed and $k \in \N$ be sufficiently large, such that 
  \[ \tilde \delta_{x,k-1} < \min \left\{\delta_0 ,(6C_2(m)+2C_x)^{-1}\right \}.\]
 Note that $\tilde \delta_{x,k-1}< \min\{\delta_0 ,(6C_2(m)+2C_x)^{-1}\}$ already implies $\delta_{x,i}\leq \tilde \delta_{x,k-1}< C_x^{-1}$ for all $i \geq k$, i.e. $k\geq \tilde k_x$. The $\delta_0$ stated in the remark after Lemma 
 \ref{lem:reifenbergfunction} already guarantees $\delta_0< \frac 1 {11}$. Moreover, we have for all $r \in (0,r_{x,k}]$
  \[ \frac 1 r \dist_{\HM}\Big( \Sigma \cap B_r(y), (y+P_y) \cap B_r(y) \Big) \leq \tilde \delta_{x,k-1} < \delta_0 \Foa y \in \Sigma \cap \overline{B_{r_{x,k}}} \subset \Sigma \cap B_{r_{x,k-1}}(x). \]
 This implies $r_{x,k} \leq R_0(x,\tilde \delta_{x,k-1}, \Sigma)$. Therefore we have
 \[k\geq \tilde k_x, \, \, \ r_{x,k} < R_0(x, \tilde \delta_{x,k-1},\Sigma) \AND \tilde \delta_{x,k-1} < \min \left\{\frac 1 {11},\delta_0 ,(6C_2(m)+2C_x)^{-1}\right \}.\]
 Lemma \ref{lem:projectionsurj} implies
  \[ (x+P_x) \cap B_{\frac r 2}(x) \subset \pi_{x+P_x}(\Sigma \cap B_r(x)) \Foa r \leq \frac  {r_{x,k}} 2.\]
 Because of $\tilde \delta_{x,k}< \frac 1{11} $, Lemma \ref{lem:pyx=pyz} yields for $r \leq \frac{r_{x,k}}2$
  \[ \sphericalangle(P_x,P_y) \leq C_2(m) \tilde \delta_{x,k} \Foa y \in B_r(x). \]
 For $y\neq y' \in \Sigma \cap B_r(x)$, there exist an $i \geq k$ with $r_{x,i+1}< \vert y'-y\vert \leq r_{x,i}$ and therefore $y' \in \Sigma \cap B_{r_{x,k}}(x) \cap B_{r_{x,i}}(y)$.
 This implies
 \begin{align*}
  \vert \pi_{P_x}^\perp (y-y')\vert &\leq \sphericalangle(P_x,P_y) \vert y-y'\vert + \vert \pi_{P_y}^\perp (y-y')\vert \\
  &\leq C_2(m) \tilde \delta_{x,k} \vert y- y' \vert + \tilde \delta_{x,i} r_{x,i}\\
  &\leq \left( C_2(m) \tilde \delta_{x,k}+ C_x \tilde \delta_{x,i} \right) \vert y-y'\vert  \\
  & < \frac 1 2\vert y-y' \vert.
 \end{align*}
 Here we have used $\tilde \delta_{x,i}\leq \tilde \delta_{x,k} < (6C_2(m)+2C_x)^{-1} \leq (2C_2(m)+2C_x)^{-1}$. Then for $\Sigma_1:=\Sigma \cap B_r(x) \cap \pi_{x+P_x}^{-1}(B_{\frac r 2}(x))$, the projection $\pi_{P_x  \mid \Sigma_1}$ is injenctive and 
  \[ \pi_{x+P_x \mid \Sigma_1} \colon \Sigma_1 \to (x+P_x) \cap B_{\frac r 2 }(x)  \]
 is bijective. We move $x$ to zero and let $\tilde \Sigma_1:= (\Sigma-x) \cap B_r(0) \cap \pi_{P_x \mid \Sigma-x}^{-1} (B_{\frac r 2 }(0))$, then the projection
  \[ \pi_{P_x \mid \tilde \Sigma_1} \colon \tilde \Sigma_1 \to P_x \cap B_{\frac r 2 }(0) \]
 is also a bijection and invertible. Especially, for all $y \in \Sigma_1$, there exists exactly one $z=z(y) \in P_x \cap B_{\frac r 2}(0)$ with
  \[ \pi_{P_x}(y-x)= z. \]
 Moreover, we have
 \begin{align*}
  y&= x + \pi_{P_x}(y-x)+ \pi_{P_x}^\perp (y-x) \\
  &= x + z + \pi_{P_x}^\perp (y-x).
 \end{align*}
 Defining 
  \[ f \colon P_x \cap B_{\frac r 2}(0) \to P_x^\perp;\ \ z \mapsto \pi_{P_x}^\perp \circ \left( \pi_{P_x \mid \tilde \Sigma_1} \right)^{-1}_{\mid P_x \cap B_{\frac r 2}(0)}(z), \]
 then we get
  \[ \pi_{P_x}^\perp(y-x) = f(z) \AND f(0)=0, \]
 because $z(x)=0$.\\
 For $z,z'\in P_x \cap B_{\frac r 2}(0)$ define
  \[ \left( \pi_{P_x \mid \tilde \Sigma_1} \right)^{-1}(z)=:y \AND \left( \pi_{P_x \mid \tilde \Sigma_1} \right)^{-1}(z')=:y'.  \]
 Now we have
 \begin{align*}
 \left \vert \left( \pi_{P_x \mid \tilde \Sigma_1} \right)^{-1}(z) - \left( \pi_{P_x \mid \tilde \Sigma_1} \right)^{-1}(z')\right \vert &= \vert y-y'\vert \\
  &\leq \vert \pi_{P_x}(y-y')\vert + \vert \pi_{P_x}^\perp(y-y')\vert \\
  &\leq \vert z-z'\vert + \frac 1 2 \vert y - y'\vert.
 \end{align*}
 This leads to 
  \[ \vert y - y' \vert \leq 2 \vert z- z'\vert, \]
 which implies the continuity of $(\pi_{P_x\mid \tilde \Sigma_1})^{-1}$ and therefore also of $f$.\\
 For $z \in P_x \cap B_{\frac r 2} (0)$ the definition of $f$ and Lemma \ref{lem:abstandebene} lead to
 \begin{align*}
  \vert f(z) \vert &= \vert \pi_{P_x}^\perp (y(z)-x) \vert \\
  &= \dist(y(z),x+ P_x)\\
  &\leq w_x( \vert y(z)-x \vert) \cdot \vert y(z)-x \vert,
 \end{align*}
 where $y(z)$ denotes the unique element of $\Sigma_1$ with $\pi_{P_x}(y(z)-x)=z$. We further get
 \begin{align*}
  \vert y(z)-x \vert &= \vert x + z +f(z) -x \vert \\
  &= \vert z +f(z) \vert \\
  &\leq \vert z \vert + \vert f(z) \vert\\
  &\leq \vert z \vert + w_x(\vert y(z)-x \vert) \cdot \vert y(z) -x \vert.
 \end{align*}
 Note that $w_x(\vert y(z)-x\vert ) \leq \tilde \delta_{x,k}< \frac 1 {11}$ and therefore
  \[ \vert y(z)-x \vert \leq \frac {11}{10} \vert z \vert. \]
 Finally, this leads to
  \[ \vert f(z) \vert \leq \frac {11}{10} w_x( \vert y(z) -x \vert) \cdot \vert z \vert = o (\vert z\vert), \]
 because $y(z) \xrightarrow[z \to 0]{}x$ and $w_x(r) \xrightarrow[r \to 0]{}0$.
 This yields the existence of $Df(0)$ and $Df(0)=0$.\\
 Let $z \in P_x \cap B_{\frac r 2}(0)$ and $F$ be defined as $F(z)=x+z+f(z)$, as well as
  \[ L:= \left( \pi_{P_x \mid P_{F(z)}} \right)^{-1} \colon P_x \to P_{F(z)}. \]
 Note that $F(z) \in B_r(x)$ and
  \[ \sphericalangle\left(P_x,P_{F(z)}\right)< C_2(m) \tilde \delta_{x,k} <\frac 1 6 <1, \]
 then %lemma 2.2 of \cite{kolasinski-strzelecki-vdm_2015} 
 Lemma \ref{lem:projectioninv} implies, that $L$ is well-defined. For $z,z+h \in P_x \cap B_{\frac r 2}(0)$, we get
  \[ F(z+h)-F(z) = L(h)+F(z+h)-F(z) -L(h). \]
 Using $e:=F(z+h)-F(z)-L(h)$ leads to
 \begin{align*}
  \pi_{P_x}(e)&= \pi_{P_x}(x+z+h+f(z+h)-x-z-f(z)-L(h))\\
  &=\pi_{P_x}(h+f(z+h)-f(z)-L(h))\\
  &=h-\pi_{P_x}(f(z+h))-\pi_{P_x}(f(z)) - \pi_{P_x}(L(h))\\
  &=h-h\\
  &=0,
 \end{align*}
 since $f(\cdot) \in P_x^\perp$ and $\pi_{P_x} \circ L = id_{P_x}$. This implies
 \begin{align*}
  \vert e \vert &= \vert \pi_{P_x}^\perp (e) \vert \\
  &\leq \sphericalangle\left(P_x,P_{F(z)}\right) \vert e \vert + \vert  \pi_{P_{F(z)}}^\perp (e)\vert\\
  &\leq C_2(m) \tilde \delta_{x,k} \vert e \vert + \vert  \pi_{P_{F(z)}}^\perp (e)\vert.
 \end{align*}
 Transforming this inequality and using $C_2(m) \tilde \delta_{x,k}<\frac 1 6$ yield
 \begin{align*}
  \vert e \vert &< \frac 6 5 \vert \pi_{P_{F(z)}}^\perp(e)\vert\\
  &= \frac 6 5 \vert \pi_{P_{F(z)}}^\perp (F(z+h)-F(z)-L(h))\vert \\
  &= \frac 6 5 \vert \pi_{P_{F(z)}}^\perp (F(z+h) -F(z) ) \vert \\
  &= \frac 6 5 \dist(F(z+h),F(z)+P_{F(z)}) \\
  &\leq \frac 6 5 w_x(\vert F(z+h)-F(z)\vert) \cdot \vert F(z+h)-F(z)\vert.
 \end{align*}
 For the last inequality we used Lemma \ref{lem:abstandebene} and the fact that $F(z),F(z+h) \in B_{r_{x,k}}(x)$, as well as $F(z+h) \in B_{r_{x,k}}(F(z))$ 
 for all $h \in P_x$ such that $z+h \in P_x \cap B_r(0)$.\\
 To estimate $\vert F(z+h)-F(z) \vert$ note
 \begin{align*}
  \vert L(h) -h \vert &= \vert\pi_{P_{F(z)}} (L(h)) - \pi_{P_x}(L(h))\vert \\
  &\leq \sphericalangle\left(P_{F(z)},P_x\right) \vert L(h)\vert\\
  &< \frac 1 6 \vert L(h) \vert.
 \end{align*}
 Therefore we get 
  \[ \frac 5 6 \vert L(h) \vert < \vert h \vert < \frac 7 6 \vert L(h) \vert. \]
 Using these estimates yields
 \begin{align*}
  \vert F(z+h)-F(z) \vert &= \vert L(h) + e \vert \\
  &\leq \vert L(h) \vert + \vert e \vert\\
  &\leq \frac 6 5 \vert h \vert + \frac 6 5 w_x(\vert F(z+h)-F(z)\vert) \cdot \vert F(z+h)-F(z)\vert.
 \end{align*}
 The fact that $F(z+h) \in B_{r_{x,k}}(F(z))$ for $z+h \in P_x \cap B_{\frac r 2}(0)$ leads to
  \[ w_x(\vert F(z+h)-F(z) \vert) \leq \tilde \delta_{x,k} < \frac 1 {11}. \]
 This implies
  \[  \vert F(z+h) -F(z) \vert < \frac {66}{49} \vert h \vert. \]
 Finally we get with the continuity of $F$
 \begin{align*}
  \vert F(z+h)-F(z) -L(h)\vert &=\vert e \vert \\
  &\leq \frac 6 5 w_x(\vert F(z+h)-F(z) \vert ) \cdot \vert F(z+h)-f(z) \vert\\
  & \leq  2 w_x(\vert F(z+h)-F(z) \vert ) \cdot \vert h \vert\\
  &= o(\vert h \vert).
 \end{align*}
 This is the differentiability of $F$ with $DF(z)=(\pi_{P_x \mid P_{F(z)}})^{-1}$ and, equivalent to this, the differentiability of $f$ with $Df(z)=DF(z)-id$.\\
 To see that $z \mapsto Df(z)$ is continuous, let $a \in P_x \cap \S^{m-1}$ and $w,z \in P_x \cap B_r(0)$, then
 \begin{align*}
  \vert (Df(z)-Df(w))a \vert &= \vert (DF(z)-DF(w))a \vert \\
  &= \vert \pi_{P_{F(z)}}(DF(z)a)- \pi_{P_{F(w)}}(DF(w)a) \vert\\
  &\leq \vert \pi_{P_{F(z)}}(DF(z)a) - \pi_{P_{F(w)}}(DF(z)a) \vert + \vert \pi_{P_{F(w)}}(DF(z)a-DF(w)a)\vert\\
  &\leq \sphericalangle\left(P_{F(z)},P_{F(w)}\right) \vert DF(z)a\vert + \vert \pi_{P_{F(w)}}(DF(z)a-DF(w)a)\vert.
 \end{align*}
 First we get
  \[ \sphericalangle\left(P_{F(z)},P_{F(w)}\right) \vert DF(z)a \vert \leq 2 C_2(m) \tilde \delta_{x,k} \vert Df(z)a+a \vert \]
 and since $Df(\cdot)a \in P_x^\perp$
 \begin{align*}
  \vert \pi_{P_{F(w)}}(DF(z)a-DF(w)a)\vert &=\vert \pi_{P_{F(w)}}(Df(z)a-Df(w)a)\vert \\
  &=\vert (\pi_{P_{F(w)}}-\pi_{P_x})(Df(z)a-Df(w)a)\vert \\
  &\leq  C_2(m) \tilde \delta_{x,k} \vert Df(z)a-Df(w)a\vert.
 \end{align*}
 In the case $w=0$ we get $Df(0)=0$ which leads to
 \begin{align*}
   \vert Df(z)a\vert &\leq 2 C_2(m) \tilde \delta_{x,k} \vert Df(z)a+a\vert +  C_2(m) \tilde \delta_{x,k} \vert Df(z)a \vert \\
 &\leq  3C_2(m) \tilde \delta_{x,k} \vert Df(z)a\vert + 2 C_2(m) \tilde \delta_{x,k}.
 \end{align*}
 Using $3 C_2(m) \tilde \delta_{x,k}<\frac 1 2$ yields
  \[ \vert Df(z)a\vert < 1 \AND \vert DF(z)a\vert < 2. \]
 Let $\varepsilon>0$ be arbitrary. There exists an $i \in \N$ such that $\tilde \delta_{x,i}<\frac 5 {12C_2(m)} \varepsilon$. Using the continuity of $F$ yields 
 the existence of an $r'>0$, such that for $w \in P_x \cap B_r(0)$ with $\vert z-w \vert<r'$, we get
  \[ \vert F(z)-F(w) \vert \leq \frac 1 2 r_{x,i} ,\Fo i \in \N_{\geq k}. \]
 This allows to improve the estimate of the angle, using Lemma \ref{lem:pyx=pyz} yields
  \[ \sphericalangle\left(P_{F(z)},P_{F(w)}\right) \leq C_2(m) \delta_{x,i} .\]
 Then the previous estimates imply
 \begin{align*}
  \vert Df(z)a -Df(w)a \vert &\leq  C_2(m) \tilde \delta_{x,i} \vert DF(z)a\vert + C_2(m) \tilde \delta_{x,k} \vert Df(z)a-Df(w)a\vert\\
  &<2 C_2(m) \tilde \delta_{x,i} + \frac 1 6 \vert Df(z)a-Df(w)a\vert.
 \end{align*}
 Finally this gives
 \[ \vert Df(z)a-Df(w)a\vert < \frac {12} 5 C_2(m) \tilde \delta_{x,i} <\varepsilon. \]
 Since we can choose $\varepsilon>0$ arbitrary, this is the continuity of $z \mapsto Df(z)$.\\
 To finish the proof let $\varphi \in C_0^\infty(P_x \cap B_{\frac r 2}(0)$ be a cut-off function with $0\leq \varphi \leq 1$ and $\varphi_{\mid P_x \cap B_{\frac r 3}(0)} \equiv 1$. 
 Define
  \[ \tilde f \colon P_x \to P_x^\perp: \ z \mapsto \begin{cases} 
                                                     \varphi(z) f(z) & \Fo z \in P_x \cap B_{\frac r 2}(0),\\
                                                     0 &\,\,\,\text{otherwise. }
                                                    \end{cases}
 \]
 Then for all $z \in P_x \cap B_{\frac r 3}$ we have $\tilde f(z)=f(z).$ Moreover, for $y \in \Sigma \cap B_{\frac r 3}(x)$ we have 
  \[ \vert \pi_{x+P_x}(y)-x \vert = \vert x + \pi_{P_x}(y-x)-x \vert < \frac r 3< \frac r 2,\]
 which implies 
 \begin{align*}
  \Sigma \cap B_{\frac r 3 }(x) &= x + \left( \operatorname{graph}(f) \cap B_{\frac r 3}(0) \right) \\
  &= x + \left( \operatorname{graph}(\tilde f) \cap B_{\frac r 3}(0) \right).
 \end{align*}

\end{proof}

To prove that every $C^1$-submanifold satisfies $(RPC)$ we will first state, that every graph of a function with bounded Lipschitz-constant can be locally approximated by planes, with respect to the Hausdorff-distance, i.e. it is Reifenberg-flat. 
The quality of this approximation is given by the Lipschitz-constant.

\begin{lem} \label{lem:graphreifenberg}
 Let $\Sigma \subset \R^n$. Assume for $x \in \Sigma$ exist a plane $P \in G(n,m)$, a radius $R>0$ and a function $u_x \colon P \to P^\perp$ with $u_x(0)=0$, $\operatorname{Lip}(u_{x\mid B_{R}(x)})\leq \alpha$, such that
  \[ \left(\Sigma \cap B_R(x)\right)-x = \operatorname{graph}(u_x) \cap B_R(x), \]
 then for all $y \in \Sigma \cap B_{\frac R 2 }(x)$ we have
  \[ \dist_{\HM}\Big( \Sigma \cap B_r(y), (y+P) \cap B_r(y) \Big) \leq r \alpha \Foa r \in \left(0, R/ 2\right]. \]
\end{lem}

\begin{proof}
 For all $y \in \Sigma \cap B_r(x)$ and $z(y)= \pi_P(y-x)$ we have
 \begin{align*}
   y&=x+ \pi_P(y-x) + \pi_P^\perp(y-x)\\
   &= x + z(y) + u_x\left(z(y)\right).
 \end{align*}
 Let $r \in (0, \frac R 2]$ be fixed. For $y \in \Sigma \cap B_{\frac R 2 }(x)$ and $\tilde y \in \Sigma \cap B_r(y)$ we get with $\pi_P(\tilde y -y ) +y \in (y+P) \cap B_r(y)$
 \begin{align*}
  \dist\Big(\tilde y , (y+P) \cap B_r(y) \Big) &\leq \left \vert \pi_P^\perp(\tilde y -y) \right \vert\\
  &= \left \vert \pi_P^\perp(\tilde y - x) - \pi_P^\perp(y-x)\right \vert\\
  &= \vert u_x(z(\tilde y)) - u_x(z(y)) \vert\\
  %&\leq \alpha \vert z(\tilde y) - z(y)\vert\\
  &\leq \alpha r.
 \end{align*}
 Note that 
 \begin{align*}
  y+P = x+z(y)+u_x(z(y)) + P =x+ u_x(z(y)) + P.
 \end{align*}
 Using $P \cap (B_r(y)-y) \subset P \cap B_R(0)$ we can write $\Sigma \cap B_r(y) = x+\operatorname{graph}(u_x) \cap B_r(y)$.
 For $x+ \tilde z + u_x(z(y)) \in (y+ P) \cap B_{\frac r {\sqrt {1+\alpha^2}}}(y)$, i.e. $\tilde z \in P \cap B_{\frac r {\sqrt {1+\alpha^2}}}(z(y))$ we have 
 \begin{align*}
  \vert x + \tilde z +u_x(\tilde z ) - y \vert &= \vert \tilde z +u_x(\tilde z) + z(y) + u_x(z(y)) \vert\\
  &= \sqrt{\vert \tilde z - z(y)\vert^2 + \vert u_x(\tilde z) - u_x(z(y)) \vert^2 }\\
  &\leq \sqrt{1+\alpha^2} \cdot \vert \tilde z -z(y) \vert\\
  &< r.
 \end{align*}
 This implies
 \begin{align*}
  \dist\Big(x+\tilde z +u_x\big(z(y)\big), \Sigma \cap B_r(y)\Big) &\leq \vert x + \tilde z +u_x(z(y)) - x -\tilde z - u_x(\tilde z)\vert\\
  &=\vert u_x(z(y))-u_x(\tilde z) \vert \\
  &\leq \frac {\alpha r}{\sqrt{1+\alpha^2}}.
 \end{align*}
 For $z' \in P \cap (B_r(z(y)) \setminus B_{\frac r {\sqrt {1+\alpha^2}}}(z(y))$ there exists a $\hat z \in P \cap B_{\frac r {\sqrt{1+\alpha^2}}}(z(y))$ with
  \[ \vert z'-\hat z \vert < \left( 1 - \frac 1 {\sqrt{1+\alpha^2}} \right)r. \]
 This leads to
  \[ \dist\Big(x+z'+u_x(z(y)), \Sigma \cap B_r(y)\Big)\leq \sqrt{\left(1-\frac 1{\sqrt{1+\alpha^2}}\right)^2 + \left( \frac \alpha {\sqrt{1+\alpha^2}}\right)^2} r \leq \alpha r. \]
 Finally this guarantees
  \[ \dist_{\HM}\Big( \Sigma \cap B_r(y) , (y+P) \cap B_r(y) \Big) \leq \alpha r. \]

\end{proof}

\begin{lem}
 An embedded $C^1$-submanifold $\Sigma$ of $ \R^n$ satisfiest $(RPC)$. Moreover, we get $P_x=T_x\Sigma$.
\end{lem}

\begin{proof}
 For all $x \in \Sigma$ and $\alpha >0$ there is a radius $\tilde R_x(\alpha)>0$ such that $(\Sigma \cap B_{\tilde R_x(\alpha)}(x))-x$ is the graph of a $C^1$-function 
 $u_x \colon T_x\Sigma \to T_x\Sigma^\perp$ with $u_x(0)=0$ and $Du_x(0)=0$ as well as $\Vert Du_x \Vert_{C^0(B_{\tilde R_x(\alpha)}(0))}\leq \alpha$. Especially $\operatorname{Lip}(u_{x \mid B_{\tilde R_x(\alpha)}}) \leq \alpha.$\\ 
 Define $R_x:=r_{x,1}:=\frac 1 2 \tilde R_x(\alpha)$. For $y \in \Sigma \cap B_{R_x}(x)$ let the plane $P(y,r_{x,1})$ be defined by 
  \[ P(y,r_{x,1}):= T_x\Sigma. \]
 Lemma \ref{lem:graphreifenberg} implies for all $y \in \Sigma \cap B_{R_x}(x)$
  \[ \dist_{\HM}\Big( \Sigma \cap B_r(y), \big(y+ P(y,r_{x,1})\big) \cap B_r(y) \Big) \leq \alpha r \Foa r \leq r_{x,1}. \]
 Now define
  \[  \delta'_{x,i}:=  \frac {\delta_{x,1}}{2^{i-1}}:=\frac \alpha {2^{i-1}} . \]
 For all $i \in \N_{>0}$ we have
  \[ \Sigma \cap \overline{B_{R_x}(x)} \subset \bigcup_{y \in \Sigma \cap \overline{B_{R_x}(x)}} {B_{\frac {\tilde R_y( \delta'_{x,i})} 2}}(y). \]
 Then there exists an $N \in \N$ and $y_1,\dots,y_N \in \Sigma \cap \overline{B_{R_x}(x)}$ with
  \[ \Sigma \cap \overline{B_{R_x}(x)} \subset \bigcup_{j=1}^N B_{\frac {\tilde R_{y_j}(\delta'_{x,i})}2} (y_j). \]
 Define $ r'_{x,1}:=r_{x,1}$ and recursively
  \[  r'_{x,i} := \min \left \{ \min_{j \in \{1,\dots,N(i)\}} \left\{\frac {\tilde R_{y_j}( \delta'_{x,i})} 2 \right\}, \frac { r'_{x,i-1}}2 \right \}, \]
 as well as $P(y, r'_{x,i}):= T_{y_j}\Sigma$ for an arbitrary $j\in \{1,\dots,N(i)\}$ with $y \in B_{\frac {\tilde R_{y_j}( \delta'_{x,i})}2}(y_j)$.\\
 Using Lemma \ref{lem:graphreifenberg} for $R= \tilde R_{y_j}(\delta_{x,i}')$, we get for all $y \in B_{r_{x,i'}}(y_j)$
  \[ \dist_{\HM} \Big( \Sigma \cap B_r(y), \big(y+P(y, r'_{x,i})\big) \cap B_r(y)  \Big) \leq  \delta'_{x,i}r \Foa r \leq  r'_{x,i}. \]
 The $B_{\tilde R_{y_j}(\delta'_{x,i})}(y_j)$ cover $\Sigma \cap B_{R_x}(x)$ and therefore we have 
  \[ \dist_{\HM} \Big( \Sigma \cap B_r(y), \big(y+P(y, r'_{x,i})\big) \cap B_r(y)  \Big) \leq  \delta'_{x,i}r \Foa r \leq r'_{x,i} \AND y \in \Sigma \cap B_{R_x}(x). \]
 This holds for all $i \in \N$. Moreover, for all $\delta>0$ there exists an $i\in \N$ with $\delta'_{x,i}<\delta$, which implies that $\Sigma$ is Reifenberg-flat with 
 vanishing constant. Note that it is important, that the $r'_{x,i}$ are independent of $y \in \Sigma \cap B_{R_x}(x)$. \\
 It remains to show that we can define a sequence of radii $r_{x,i}$ which is controlled by a constant $C_x$, as well as the convergence of the planes $P(y,r_{x,i})$ to 
 $P_y=T_y\Sigma$.\\
 To see this, note that Lemma \ref{lem:metricinequ} implies
  \[ \sphericalangle\left(T_y\Sigma , P(y, r'_{x,i})\right) = \sphericalangle\left( T_y\Sigma , T_{y_j}\Sigma\right) \leq \delta'_{x,i} \Foa y \in \Sigma \cap B_{R_x}(x). \]
 This yields
  \[ \sup_{y \in B_{R_x}(x)} \sphericalangle\left(T_y \Sigma, P(y, r'_{x,i})\right) \leq  \delta'_{x,i} \xrightarrow [i \to \infty]{} 0. \]
 Now let $C_x>1$ be fixed. For all $i \in \N$, there exists an $l=l(i) \in \N_0$ with
  \[ C_x^l r'_{x,i+1} <  r'_{x,i} \leq C_x^{l+1} r'_{x,i+1}. \]
 If $r_{x,s}= r'_{x,i}$ and $\delta_{x,s}=\delta'_{x,i}$ 
 are defined, set recursively %$r_{x,j+1},\dots,r_{x,j+l(i)+1}$ as
 \begin{align*}
  r_{x,s+k}&:= \frac 1 {C_x^k} r_{x,s} \Fo k \in \{ 1,\dots, l(i)\}\\
  r_{x,s+l+1}&:=  r'_{x,i+1},
 \end{align*}
  \[ P(y,r_{x,s+k}):= P(y,r_{x,s}) = P(y, r'_{x,i}) \Fo k \in \{ 1,\dots, l(i)\} \]
 and 
 \begin{align*}
  \delta_{x,s+k}&:=\delta_{x,i} \Fo k \in \{ 1,\dots, l(i)\},\\
  \delta_{x,s+l(i)+1}&:=  \delta'_{x,i+1}.
 \end{align*}
 These definitions lead to
  %\[ r_{x,j+1} < r_{x,j} \leq C r_{x,j+1} \]
  \[ \sup_{y \in B_{R_x}(x)} \dist_{\HM} \Big(\Sigma \cap B_{r_{x,s}}(y), \big(y+P(y,r_{x,s})\big) \cap B_{r_{x,s}}(y) \Big) \leq \delta_{x,s} r_{x,s} \Foa s \in \N \]
 with $\lim_{s \to \infty} \delta_{x,s} =0$
 and 
  \[ \sup_{y \in B_{R_x}(x)} \sphericalangle\left(T_y\Sigma, P(y,r_{x,s})\right) \leq \varepsilon_{x,i}:= \delta_{x,s}. \]
 Moreover, if $s\in \N$ such that $r_{x,s}=r'_{x,i}$, then the definition of $r_{x,s}$ leads to
 \begin{align*}
  \frac {r_{x,s+k}}{r_{x,s+k+1}} &= C_x \Fo k \in \{0,\dots,\max\{0,l(i)-1\}\}\\
  \frac {r_{x,{j+l(i)}}}{r_{x,j+l(i)+1}}&= \frac {r'_{x,i} \cdot \frac 1 {C_x^{l(i)}}}{r'_{x,i+1}} \leq \frac {C_x^{l(i)+1}}{C_x^{l(i)}} = C_x.
 \end{align*}
 Finally these are all conditions required for $\Sigma$ to satisfy $(RPC)$.
\end{proof}
%%%%%%%%%%%%%%%%%%%%%%%%%%%%%%%%%%%%%%%%%%%%%%%%%%%%%%%%%%%%%%%%%%%%%%%%%%%%%%%%%%%%%%%%%%%%%%%%%%

\section{\bf{Proof of Theorem \ref{thm:intcondition}}}\label{sec:intcon_proof}

Unlikely Toro's condition in (2), the integral condition postulated in Theorem \ref{thm:intcondition} does not need a small bound but only to be finite. 
Note that the important part of this condition is the decay of $\theta_{B_{R_x}(x)}$ near zero, i.e. if for $x \in \Sigma$ there exists an $R_x > 0$ with
 \[\int \limits_0^{R_x} \frac {\theta_{ B_{R_x}(x)}(r)} r \ dr < \infty,\]
then for all $r,R$ with $0<r \leq R_x \leq R < \infty$ we get
\begin{align*}
 \int \limits_0^r \frac {\theta_{ B_{R_x}(x)}(r)} r \ dr &\leq \int \limits_0^{R} \frac {\theta_{ B_{R_x}(x)}(r)} r \ dr 
 \\&= \int \limits_0^{R_x} \frac {\theta_{ B_{R_x}(x)}(r)} r \ dr  + \int \limits_{R_x}^{R} \frac {\theta_{ B_{R_x}(x)}(r)} r \ dr \\
 & \leq \int \limits_0^{R_x} \frac {\theta_{ B_{R_x}(x)}(r)} r \ dr + \int \limits_{R_x}^{R} \frac {1} r \ dr\\
 &< \infty.
\end{align*}
On the other hand, we can not expect $R_x$ to contain any information about the size of the graph patches for $\Sigma$.\\
We will prove Theorem \ref{thm:intcondition} by showing that each $\Sigma$, which has an finite integral already satisfies $(RPC)$.

\begin{proof}[\textbf{Proof of Theorem \ref{thm:intcondition}}]  
 Let $C>1$ be arbitrary. For every $k \in \N$ there exist an $r_{x,k} \in (R_x/ C^{\frac {k+1 }2},R_x/ C^{\frac {k }2})$ with
  \[ \frac {\theta_{B_{R_x}(x)}(r_{x,k})} {r_{x,k}} \leq \int \limits_{R_x/ C^{\frac {k+1 }2}}^{R_x/ C^{\frac {k }2}} \frac {\theta_{B_{R_x}(x)}(r)} r \ dr \cdot \frac 1 {R_x \left(C^{-\frac {k }2} - C^{-\frac {k+1 }2} \right)}, \]
 otherwise we would get
 \begin{align*}
  \int \limits_{R_x/ C^{\frac {k+1 }2}}^{R_x/C^{\frac {k }2}} \frac {\theta_{B_{R_x}(x)}(r)} r \ dr 
  &>  \int \limits_{R_x/ C^{\frac {k+1 }2}}^{R_x/ C^{\frac {k }2}} \frac 1 {R_x \left(C^{-\frac {k }2} - 
  C^{-\frac {k +1}2} \right)} \int \limits_{R_x/ C^{\frac {k+1 }2}}^{R_x/ C^{\frac {k }2}} \frac {\theta_{B_{R_x}(x)}(r')} {r'} \ dr' \ dr \\
  &=\int \limits_{R_x/C^{\frac {k+1 }2}}^{R_x/ C^{\frac {k }2}} \frac {\theta_{B_{R_x}(x)}(r')} {r'} \ dr',
 \end{align*}
 which is a contradiction. Therefore, we have
  \[ r_{x,k+1}<r_{x,k} \leq C r_{x,k+1} \AND \lim_{k\to \infty} r_{x,k}=0. \]
 Moreover
 \begin{align*}
  \theta_{B_{R_x}(x)}(r_{x,k})&\leq  \frac {r_{x,k}} {R_x \left(C^{-\frac {k }2} - C^{-\frac {k +1}2} \right)} \cdot \int \limits_{R_x/ C^{\frac {k +1}2}}^{R_x/ C^{\frac {k }2}} \frac {\theta_{B_{R_x}(x)}(r)} r \ dr \\
  &\leq  \frac {R_xC^{-\frac {k }2}} {R_x C^{-\frac {k }2} \left(1 - C^{-\frac {1 }2} \right)} \cdot \int \limits_{R_x/ C^{\frac {k+1 }2}}^{R_x/ C^{\frac {k }2}} \frac {\theta_{B_{R_x}(x)}(r)} r \ dr\\
  &= \frac {C^{\frac 1 2}}{ C^{\frac 1 2} -1}  \cdot \int \limits_{R_x/ C^{\frac {k+1 }2}}^{R_x/ C^{\frac {k }2}} \frac {\theta_{B_{R_x}(x)}(r)} r \ dr.
 \end{align*}
 Therefore
 \begin{align*}
  \sum \limits_{k=0}^\infty \theta_{B_{R_x}(x)}(r_{x,k}) &\leq \frac {C^{\frac 1 2}}{ C^{\frac 1 2} -1}  \cdot \sum \limits_{k=0}^\infty \int \limits_{R_x/ C^{\frac {k+1 }2}}^{R_x/ C^{\frac {k }2}}  \frac {\theta_{B_{R_x}(x)}(r)} r \ dr\\
  &\leq \frac {C^{\frac 1 2}}{ C^{\frac 1 2} -1}  \int \limits_{0}^{R_x} \frac {\theta_{B_{R_x}(x)}(r)} r \ dr\\
  &<\infty.
 \end{align*}
 For $\delta_{x,k}:=\theta_{B_{R_x}(x)}(r_{x,k})$, this implies
  \[ \delta_{x,k} \xrightarrow[k\to \infty]{}0. \]
 Then we get for all sufficiently large $k \in \N$
  \[ \frac 2 {1-2\delta_{x,k+1}} (\delta_{x,k+1}+2C \delta_{x,k}) < \tilde C (\delta_{x,k+1}+2C \delta_{x,k}) < \frac 1 {\sqrt2}. \]
 Let $P(y,r_{x,k})$ denote a plane which approximates $\Sigma$ at $y \in \Sigma \cap B_{R_x}(x)$ and scale $r_{x,k}$, corresponding to $\delta_{x,k}$. Then Lemma \ref{lem:reifenbergwinkel} leads to
  \[ \sphericalangle\big(P(y,r_{x,k}),P(y,r_{x,k+1})\big) \leq \tilde C C_1(m) (\delta_{x,k+1} + 2 C \delta_{x,k}).\]
 For $i \in \N$ we get
 \begin{align*}
  \sphericalangle\big(P(y,r_{x,k}),P(y,r_{x,k+i})\big) &\leq \sum \limits_{l=0}^{i-1} \sphericalangle\big(P(y,r_{x,k+l}),P(y,r_{x,k+l+1})\big)\\
  &\leq \tilde C C_1(m) \sum \limits_{l=0}^{i-1} (\delta_{x,k+l+1}+2 C \delta_{x,k+l})\\
  &\xrightarrow[k \to \infty]{} 0,
 \end{align*}
 since $\sum \limits_{k=1}^\infty \delta_{x,k} < \infty.$
 This yields the existence of a plane $P_y \in G(n,m)$ such that 
  \[ \sphericalangle\big(P(y,r_{x,k}),P_y\big) \xrightarrow [k \to \infty]{}0. \]
 In particular, for all $\varepsilon >0$ there exist a $J_y \in \N$ such that
  \[ \sphericalangle\big(P(y,r_{x,k}),P_y\big) < \varepsilon \Foa k \geq J_y. \]
 For $i \in \N$ and $k >\max\{i,J_y\}$ we get
 \begin{align*}
  \sphericalangle\big(P(y,r_{x,i}),P_y\big) &\leq \sphericalangle\big(P(y,r_{x,i}),P(y,r_{x,k})\big) + \sphericalangle\big(P(y,r_{x,k}), P_y\big)\\
  &\leq \sum \limits_{l=0}^{k-i-1} \sphericalangle\big(P(y,r_{x,i+l}),P(y,r_{x,i+l+1})\big) + \varepsilon\\
  &\leq \sum \limits_{l=0}^{\infty} \sphericalangle\big(P(y,r_{x,i+l}),P(y,r_{x,i+l+1})\big) + \varepsilon.
 \end{align*}
 The limit $\varepsilon \to 0$ yields
 \begin{align*}
  \sphericalangle\big(P_y,P(y,r_{x,i})\big) &\leq \sum \limits_{l=0}^{\infty} \sphericalangle\big(P(y,r_{x,i+l}),P(y,r_{x,i+l+1})\big) \\
  &\leq \tilde C C(m) \sum_{l=i}^\infty (\delta_{x,l+1}+ 2 C \delta_{x,l}),
 \end{align*}
 if $i\geq N$ and $N \in \N$ such that
  \[ \frac 2 {1-2 \delta_{x,k+1}} (\delta_{x,k+1} + 2 C \delta_{x,k})<\tilde C(\delta_{x,k+1} + 2 C \delta_{x,k}) < \frac 1 {\sqrt 2} \Foa k \geq N.\]
 Then
  \[\varepsilon_{x,k}:= \begin{cases}
                         \tilde C C(m) \sum_{l=k}^\infty (\delta_{x,l+1}+ 2 C \delta_{x,l}) & \Fo k \geq N,\\
                         1 & \, \, \, \text{otherwise},
                        \end{cases}
  \]
 is independent of $y \in B_{R_x}(x)$ with
  \[ \sphericalangle\big(P_y,P(y,r_{x,k})\big)\leq \varepsilon_{x,k} \xrightarrow [k \to \infty]{} 0. \] 
 This is the condition of $(RPC)$ for $C=C_x$ and Lemma \ref{lem:RPCisC1} finishes the proof.
\end{proof}

\begin{rem} \label{rem:intcondition}
 An immediat result of the proof is that if there exist a constant $C>0$ and a monotonically decreasing sequence $(r_{x,k})_k \subset(0,R_x]$ with 
  \[ r_{x,k} \leq C r_{x,k+1} \AND \lim_{k\to \infty} r_{x,k}=0 \]
 such that
  \[ \sum \limits_{k=1}^\infty \theta_{B_{R_x}(x)}(r_{x,k})< \infty, \]
 then $\Sigma$ is an embedded, $m$-dimensional $C^1$-submanifold of $\R^n$. Moreover, the finiteness of the integral in Theorem \ref{thm:intcondition} implies this condition. 
\end{rem}
%%%%%%%%%%%%%%%%%%%%%%%%%%%%%%%%%%%%%%%%%%%%%%%%%%%%%%%%%%%%%%%%%%%%%%%%%%%%%%%%%%%%%%

\begin{appendix}
\section{\bf{A Reifenberg-flat set with vanishing constant without $C^1$-regularity}} \label{sec:example_notC1}

Let 
\begin{align*}
 u \colon \R \to \R &, \ u(z):= \sum_{k=1}^\infty \frac {\cos(2^kz)}{2^k\sqrt k}\\
 \intertext{and}
 U \colon \R \to \R^2 &, \ U(z):= \binom z {u(z)}.
\end{align*}
Then $\Sigma := \operatorname{graph}(u)=U(\R)$ is Reifenberg-flat with vanishing constant as stated in \cite{Toro_1997}.\\
Assume $\Sigma$ is a $C^1$-submanifold of $\R^2$. Then for all $x \in \Sigma$ and all $\alpha>0$ there exists a radius $r=r(x,\alpha)>0$ and a $C^1$-function $f_x \colon T_x\Sigma \to T_x\Sigma^\perp$ such that
 \[ \Sigma \cap B_r(x) = (x + \operatorname{graph}(f_x)) \cap B_r(x) \]
and 
 \[ \Vert f_x' \Vert_{C^0(T_x\Sigma \cap B_r(0),T_x\Sigma^\perp)} \leq \alpha. \]
Due to the symmetry of $u$, i.e. $u(z)=u(-z)$ for all $z \in \R$, we have for $x_0=U(0)$
 \[ T_{x_0}\Sigma \neq \{0\} \times \R. \]
This implies that there exists an $r'>0$ with
 \[ \left( \R \times \{0\} \right) \cap B_{r'}(0) \subset \pi_{\R \times \{0\}} \left( T_{x_0}\Sigma \cap B_r(0)\right). \]
Without loss of generality let $r'$ be small enough such that $U(z) \in B_r(x_0)$ for all $z \in B_{r'}(0)$.\\
The representation as a graph of $f_{x_0}$ yields the injectivity of 
 \[ g\colon \left(\R \times \{0\}\right) \cap \overline{B_{\frac {r'}2}(0)} \to \R \times \{0\}, \ t \mapsto \pi_{\R \times \{0\}} \left( \pi_{T_{x_0}\Sigma} \left(U(t)-U(0)\right) \right).\]
Together with the continuity of $g$ this implies that $g$ is monotonic. Then for $-\frac {r'}2 =t_0<t_1<\dots<t_k=\frac {r'}2$ and $t_i':= \pi_{T_{x_0}\Sigma}(U(t_i)-U(0))$ for $i=0,\dots,k$ we get either
\begin{align*}
 \pi_{\R \times \{0\}} \left( t_0'\right)< \pi_{\R \times \{0\}} \left(t_1'\right)< \dots <  \pi_{\R \times \{0\}} \left(t_k'\right),
\intertext{or}
  \pi_{\R \times \{0\}} \left(t_0'\right)> \pi_{\R \times \{0\}} \left(t_1'\right)> \dots >  \pi_{\R \times \{0\}} \left(t_k'\right).
\end{align*}
Therefore we have $\sum_{i=1}^k \vert t_i'-t_{i-1}'\vert = \vert t_k'-t_0'\vert$ and
\begin{align*}
 \sum_{i=1}^k \vert U(t_i)-U(t_{i-1})\vert & = \sum_{i=1}^k \left\vert \binom {t_i'}{f_{x_0}(t_i')} - \binom {t_{i-1}'}{f_{x_0}(t_{i-1}')}\right \vert\\
 &\leq \sum_{i=1}^k \sqrt{1+\alpha^2} \cdot \vert t_i'-t_{i-1}' \vert\\\
 &=\sqrt{1+\alpha^2}\cdot \left \vert \pi_{T_{x_0}\Sigma}\left(U\left(-\frac {r'}2\right)\right) -\pi_{T_{x_0}\Sigma}\left(U\left(\frac {r'}2\right)\right) \right \vert
\end{align*}
which is independent of the partition of the intervall $[-r'/2,r'/2]$. This implies $U \in BV([-r'/2,r'/2],\R^2)$ and $u\in BV([-r'/2,r'/2])$. Then $u$ has to be differentiable for almost all $z \in [-r'/2,r'/2]$ which is a contradiction to $u$ 
being not differentiable for all $z \in \R$.

\section{\bf{Counterexample for integral condition }} \label{sec:example_proof}

 The finiteness of the integral as well as of the sum in Theorem \ref{thm:intcondition} respectively remark \ref{rem:intcondition} imply that $\Sigma$ is a $C^1$-submanifold, but the following example will show, that these conditions are not equivalent. Moreover, one can ask 
 if $C^1$-submanifolds are characterized by
  \[ \int \limits_0^1 \frac {\theta_{B_{R_x}(x)}^\beta(r)}{r^\alpha}\ dr < \infty \Foa x \in \Sigma \]
 for any $\alpha, \beta >0$. Note that as in Theorem \ref{thm:intcondition} the upper bound of the integral can be replaced by any $R>0$ and the case $\alpha=\beta=1$ leads to the situation of Theorem \ref{thm:intcondition}.\\
 Using $\theta_{B_{R_x}}(r)\leq 1$ for all $x \in \Sigma$ and $r>0$ leads
  \[ \int \limits_0^1 \frac {\theta_{B_{R_x}(x)}^\beta(r)}{r^\alpha}\ dr \leq \int \limits_0^1 \frac {1}{r^\alpha}\ dr < \infty \Foa 0<\alpha <1, \]
 which does not depend on $\Sigma$. Therefore, if such a condition exists, $\alpha$ has to be greater or equal to one.\\
 Moreover, the finiteness of the integral with $\alpha>1$ and $\beta < 1$ implies the finiteness for $\alpha,\beta=1$. For $\alpha=1$ and fixed $\beta \geq 1$, the following example will provide a set $\Sigma \subset \R^2$, 
 which is a one-dimensional $C^1$-submanifold, but yields neither a finite integral nor a finite sum of its $\theta$-numbers. 
 
\begin{ex} \label{ex:intcondition}
 Let $\beta\geq1$ and
  \[ f_\beta \colon \left(-\frac 1 2,\frac 1 2\right) \to \R, \ y \mapsto \begin{cases} \bigg(- \frac 2 {\log(y^2)}\bigg)^{\frac 1 \beta}  & \Fo y \in \R \setminus \{0\},\\
                                           0 & \Fo y=0,
                                         \end{cases}
  \]
 and 
  \[g_\beta \colon \R \to \R, \ x \mapsto \begin{cases}
                                     \int \limits_{-\frac 1 2}^0 f_\beta(y) \ dy - \frac {x+\frac 1 2} {\log(2)^{\frac 1 \beta}} &\Fo y \in (-\infty, -\frac 1 2),\\
                                     \int \limits_x^0 f_\beta(y) \ dy &\Fo y \in [-\frac 1 2,0),\\
                                     \int \limits_0^x f_\beta(y) \ dy &\Fo y \in [0,\frac 1 2 ],\\
                                     \int \limits_0^{\frac 1 2} f\beta(y) \ dy + \frac {x- \frac 1 2} {\log(2)^{\frac 1 \beta}} &\Fo y \in (\frac 1 2 ,\infty).
                                    \end{cases}
\]
 Then $f_\beta$ is a continuous function and $g_\beta$ is $C^1$, but $g \not\in C^{1,\sigma}$ for every $\sigma>0$. The set $\Sigma := \operatorname{graph}(g_\beta)$ is a 
 $C^1$-submanifold of $\R^n$.\\
 For all $r \leq 2 e^{-1}<1$ we get
  \[ \left \vert \log\left( \frac {r^2} 4 \right) \right \vert \geq 2. \]
 Therefore,
  \[ \left \vert g_\beta\left(\frac r 2  \right)\right \vert = \int \limits_0^{\frac r 2} \left(\frac 2 {\vert \log (y^2) \vert }\right)^{\frac 1 \beta}\ dy \leq \frac r 2 \cdot \left(\frac 2 {\vert \log( \frac {r^2} 4) \vert}\right)^{\frac 1 \beta}  
  \leq \frac r 2\]
 and hence $\binom {\frac r 2 } {g_\beta( \frac r 2)} \in \Sigma \cap B_r(0)$ for all $r\leq 2e^{-1}$.
 Due to the symmetry of $g_\beta$, the planes, which realise $\theta(0,r)$ have to be equal to $T_0 \Sigma= \R \times \{0\}$. For all small $r$ we get
 \begin{align*}
  \theta(0,r) &\geq \frac {g_\beta(\frac r 2)} r \\
  &= \frac 1 r \int \limits_0^{\frac r 2} \left(- \frac 2 {\log (y^2)}\right)^{\frac 1 \beta} \ dy\\
  &\geq \frac 1 r \int \limits_{\frac r 4}^{\frac r 2} \left(- \frac 2 {\log (y^2)}\right)^{\frac 1 \beta} \ dy\\
  &\geq \frac 1 r \cdot \frac r 4 \cdot \left(- \frac 1 {\log ( \frac r 4)} \right)^{\frac 1 \beta}\\
  &= \frac 1 4 \cdot \left(- \frac 1 {\log ( \frac r 4)} \right)^{\frac 1 \beta}.
 \end{align*}
 For all $R>0$ and monotonically decreasing sequences $(r_i)_{i \in \N} \subset (0,\max\{R,2e^{-1}\}]$ and $C>1$ with
  \[ r_i \leq C r_{i+1} \Foa i \in \N \]
 and therefore
  \[ r_1 \leq C^{i-1}r_i, \]
 we get 
 \begin{align*}
  \theta_{B_R(0)}^\beta(r_i) &\geq \frac 1 {4^\beta} \cdot \frac {-1}{ \log (\frac {r_i} 4)}\\
  &\geq \frac 1 {4^\beta} \cdot \frac {-1}{ \log (\frac {r_1} {4C^{i-1}})} \\
  &= \frac 1 {4^\beta} \cdot \frac {-1}{ \log( \frac {r_1} 4) - \log (C^{i-1}))}. 
 \end{align*}
 Finally 
 \begin{align*}
  \sum \limits_{i=1}^\infty \theta_{B_R(0)}^\beta(r_i) 
  &\geq \frac 1 {4^\beta} \sum \limits_{i=1}^\infty \frac 1 {- \log(\frac {r_1}4) + \log (C^{i-1})} \\
  &\geq \frac 1 {4^\beta} \sum \limits_{i=1}^\infty \frac 1 {- \log(\frac {r_1}4) + (i-1)\log (C)} \\
  &= \infty.
 \end{align*}
 Using the same argument of remark \ref{rem:intcondition}, this implies that also
  \[ \int \limits_0^R \frac {\theta^\beta_{B_R(0)}(r)} r \ dr = \infty \Fo R>0. \]
\end{ex} 

\section{\bf{Proof of Lemma \ref{lem:reifenbergfunction} and Lemma \ref{lem:projectionsurj}} } \label{sec:reifenbergproof}

\begin{proof}[\textbf{Proof of Lemma \ref{lem:reifenbergfunction}}]
  \begin{enumerate}[wide, labelwidth=!, labelindent=0pt]
  \item Notation:\\
  Define 
  \begin{align*}
   S_0&:=(x+L)\cap \overline{B_r(x)},\\
   \Sigma_x&:= \Sigma \cap \overline{B_r(x)},\\
   \tau_0 &\colon S_0 \to S_0 ; \ \ z\mapsto z,\\
   \delta_0 &<\left( 48(3C_1(m)+ 2) \right) ^{-1}
  \end{align*}
 and $R_0>0$ small enough, that for all $r \in (0,R_0]$ we get
  \[ \frac 1 r \inf_{L\in G(n,m)} \dist_{\HM} \Big( \Sigma \cap B_r(y) , (y+L) \cap B_r(y)  \Big) \leq \delta \Foa y \in \Sigma \cap \overline{B_{R_0}(x)}. \]
 For $j\in \N_0$ let 
  \[ r_j:= \frac {r}{12 \cdot 4^j}. \]
 For all $j>0$ we get
  \[ \Sigma_x \subset \bigcup_{z \in \Sigma_x} B_{{r_j}}(z). \]
 The compactness of $\Sigma_x$ implies the existence of a $k_j \in \N$ and a set $ Z_j := \{z_{j,1}, \dots, z_{j,k_j} \}$ with
  \[ \Sigma_x \subset \bigcup_{z \in  Z_j}  B_{{r_j}}(z). \]
 Moreover, there exists a partition of unity $\{ \varphi_z \}_{z \in Z_j}$ with
  \begin{align*}
   0 \leq \varphi_z(y) &\leq 1 \Foa y \in \R^n \AND z \in Z_j,\\
   \varphi_z(y) &=0 \Foa y \in \R^n \AND z \in \Z_j \With \vert y-z \vert \geq 3r_j,\\
   \sum_{z \in Z_j} \varphi_z(y) &= 1 \Foa y \in V_j:=\{y \in \R^n \mid \ \dist(y,\Sigma_x)<r_j \}. 
  \end{align*}
 Note that $V_j \subset \bigcup_{z \in Z_j} B_{3r_j}(z)$. Then the existence of this partition is an immidiate result of e.g. \cite[p. 52]{guillemin-pollack_1974}.\\
 For $z \in Z_j$ let $L(z,12r_j) \in G(n,m)$ denote a plane with
  \[ \dist_{\HM}\Big( \Sigma \cap B_{12r_j}(z), \big(z+L(z,12r_j)\big) \cap B_{12r_j}(z) \Big) \leq 12 r_j \delta. \]
 The $\delta$-Reifenberg-flatness of $\Sigma$ and the fact that
  \[ 12r_j \leq r \leq R_0 \]
 guarantees the existence of $L(z,12r_j)$.\\
 Now define
 \begin{align*}
  \sigma_j(y) &:=y - \sum_{z \in Z_j} \varphi_z(y) \cdot \pi_{L(z,12r_j)}^\perp(y-z)\\
  \intertext{and}
  \tau_j(y) &:= (\sigma_j \circ t_{j-1})(y).
 \end{align*}
 \item For $y \in V_j \cap \overline{B_{r-2r_j(1+6\delta)}(x)}$ we get 
 \begin{align*}
  \dist\left(\sigma_j(y), \Sigma_x\right) &\leq (36C_1(m) + 24) r_j\delta\\
  \intertext{and}
  \vert \sigma_j(y)-y\vert &\leq \dist(y,\Sigma_x) + (36C_1(m) + 24)r_j \delta\\
   &\leq (1+ 36C_1(m)\delta + 24\delta)r_j.
 \end{align*}
 Note that 
  \[ r-2r_j(1+6\delta) \geq r- \frac 1 6 r \left(1+\frac 1 {16}\right) >0 \Foa j \in \N_0. \]
 Let $y \in V_j\cap \overline{B_{r-2r_j(1+6\delta)}(x)}$ and $Z_j(y):= \{ z \in Z_j \mid \ \vert z-y\vert <3r_j \}.$ Then we get
  \[ \sigma_j(y)= y- \sum_{z \in Z_j(y)} \varphi_z(y) \cdot \pi_{L(z,12r_j)}^\perp(y-z).\]
 For $z,z' \in Z_j(y)$, we have $\vert z - z'\vert <6r_j = \frac {12r_j}2$. The definition of $\delta_0$ further yields
  \[ \frac 6 {1-2\delta} \delta < 12\delta < \frac 1 {\sqrt 2}.\]
 Lemma \ref{lem:reifenbergwinkel} implies for $x_1=z,x_2=z'$, $\delta_1=\delta_2=\delta$, $r_1=r_2=12r_j$ and $P_1=L(z,12r_j),P_2=L(z',12r_j)$ that
  \[ \sphericalangle\big(L(z,12r_j),L(z',12r_j)\big) \leq 12 C_1(m) \delta. \]
 For fixed $z_0 \in Z_j(y)$ such that $\vert z_0 -y \vert <2 r_j$ define
  \[ \tilde y := y -\pi_{L(z_0,12r_j)}^\perp (y-z_0) \]
 and we get
 \begin{align*}
  \vert \sigma_j(y) - \tilde y\vert &=  \left\vert \sum_{z \in Z_j(y)} \left( \varphi_z(y) \cdot  \pi_{L(z,12r_j)}^\perp (y-z)\right) - \pi_{L(z_0,12r_j)}^\perp (y-z_0) \right \vert \\
  &= \left \vert  \sum_{z \in Z_j(y)}  \varphi_z(y) \cdot \left( \pi_{L(z,12r_j)}^\perp (y-z) - \pi_{L(z_0,12r_j)}^\perp (y-z_0) \right) \right \vert \\
  &= \left \vert  \sum_{z \in Z_j(y)}  \varphi_z(y) \cdot \left( \pi_{L(z,12r_j)}^\perp (y-z) - \pi_{L(z_0,12r_j)}^\perp (y-z) - \pi_{L(z_0,12r_j)}^\perp (z-z_0)\right) \right \vert\\
  &\leq \sum_{z \in Z_j(y)}  \varphi_z(y) \cdot \left( \left \vert \pi_{L(z,12r_j)}^\perp (y-z) - \pi_{L(z_0,12r_j)}^\perp (y-z) \right \vert + \left \vert \pi_{L(z_0,12r_j)}^\perp (z-z_0)\right \vert \right)\\
  &\leq  \sum_{z \in Z_j(y)}  \varphi_z(y) \cdot \Big(12 C_1(m) \delta \cdot 3 r_j + \dist\big(z,z_0 + L(z_0,12r_j)\big) \Big)\\
  &\leq \left( 36C_1(m) + 12 \right) r_j \delta.
 \end{align*}
 In the last inequalities we used $z \in \Sigma \cap B_{12r_j}(z_0)$ and therefore $\dist(z,z_0 + L(z_0,12r_j))\leq 12r_j \delta$, as well as the fact that 
 $\sum_{z \in Z_j(y)}  \varphi_z(y)=1$ for $y \in V_j$ several times.\\
% $\Sigma$ is closed and with the $\delta$-Reifenberg-flatness we get, that there exists a $w \in \Sigma \cap B_{12r_j}(z_0) \subset \Sigma_x$ with
 $\tilde y \in L(z_0,12r_j)\cap B_{12r_j}(z_0)$ implies that there exists a $w \in \Sigma \cap B_{12r_j}(z_0) \subset \Sigma_x$ with
  \[ \vert \tilde y - w \vert \leq 12 r_j \delta. \]
 Using $\vert \tilde y - x \vert \leq \vert y-x \vert + \vert y-z_0\vert$, we get
 \begin{align*}
  \vert w-x \vert &\leq \vert w-\tilde y\vert + \vert \tilde y -x \vert\\
 % &\leq \vert w -\tilde y \vert + \vert y-x \vert + \vert y-z_0\vert\\
  &< 12r_j\delta + r-2r_j(1+6\delta) + 2r_j\\
  &= r.
 \end{align*}
 This implies $w \in \Sigma_x$ and
 \begin{align*}
  \dist\left(\sigma_j(y),\Sigma_x\right) &\leq \vert \sigma_j(y) - \tilde y\vert + \vert \tilde y - w\vert \\
  &\leq \left( 36 C_1(m) + 24 \right) r_j \delta.
 \end{align*}
 Due to the definition of $V_j$ and the fact that $\Sigma_x$ is closed,for all $y \in V_j$  we get a $w' \in \Sigma_x$ with
  \[ \dist\left(y,\Sigma_x\right)=\vert y-w'\vert < r_j. \]
 This yields 
  \[ \vert z_0 - w' \vert < 3 r_j \]
 and therefore 
 \begin{align*}
  \vert \tilde y - y \vert &= \left \vert \pi_{L(z_0,12r_j)}^\perp(y-z_0)  \right \vert\\
  &\leq \left \vert \pi_{L(z_0,12r_j)}^\perp (y-w') \right \vert + \left \vert \pi_{L(z_0,12r_j)}^\perp (w'-z_0) \right \vert\\
  &\leq \vert y - w'\vert + 12 r_j \delta.
  %&= \dist\left(y,\Sigma_x\right) + 12r_j \delta.
 \end{align*}
 Finally we get
  \[ \vert \sigma_j(y)-y\vert \leq \dist\left(y,\Sigma_x\right) + \left( 36C_1(m) + 24 \right) r_j \delta. \]
 \item
 For $y \in S_0 \cap \overline{B_{r'}(x)}$ with $r':=r-(2+36C_1(m) \delta +24 \delta)\sum \limits_{k=1}^\infty r_k$ we get 
  \[ \tau_j(y) \in V_{j+1} \cap \overline{B}_{r-(2+36C_1(m) \delta +24 \delta)\sum \limits_{k=j+1}^\infty r_k}(x) \Foa j \in \N_0.\]
 Note that 
  \[r' = r-(2+36C_1(m) \delta +24 \delta)\sum \limits_{k=1}^\infty r_k > r - \frac r {12} \cdot \frac 1 3 (2+\frac 1 4) =\frac {15}{16}r \]
 and 
  \[r' \leq r- 2r_j(1+6\delta). \]
 For $j=0$ and $y \in S_0\cap \overline{ B_{r'}(x)}$ we have $\tau_0(y)=y$ and the Reifenberg-flatness yields
  \[ \dist\left(y, \Sigma_x\right) \leq r\delta < \frac r {48}=r_1. \]
 This implies $\tau_0(y)=y \in V_1\cap \overline{ B_{r'}(x)}$.\\
 Now we assume that the statement holds for $j-1 \in \N_0$ and let $y \in S_0\cap \overline{ B_{r'}(x)}$. We have
 \begin{align*}
  \tau_{j-1}(y) &\in V_{j} \cap \overline{B}_{r-(2+36C_1(m) \delta +24 \delta)\sum \limits_{k=j}^\infty r_k}(x)\\
  &\subset V_j \cap \overline{B}_{r-r_j(2+36C_1(m)\delta+24\delta)}(x)\\
  &\subset V_j \cap \overline{B}_{r-2r_j(1+6\delta)}(x).
 \end{align*}
 Therefore step $(2)$ implies
 \begin{align*}
  \dist\left(\tau_j(y),\Sigma_x\right) &=\dist\left(\sigma_j(\tau_{j-1}(y)),\Sigma_x\right)\\
  &\leq (36C_1(m)+24) r_j\delta\\
  &< r_{j+1},
 \end{align*}
 which is $\tau_j(y) \in V_{j+1}$. Moreover, step $(2)$ leads to 
 \begin{align*}
 \vert \tau_j(y) -x \vert &\leq \vert \sigma_{j}(\tau_{j-1}(y)) - \tau_{j-1}(y) \vert + \vert \tau_{j-1}(y) -x\vert \\
 & \leq (1+36C_1(m)\delta + 24 \delta)r_j + r -(2+36C_1(m)\delta+24 \delta)\sum \limits_{k=j}^\infty r_k\\
 &\leq r- (2+36C_1(m)\delta+24 \delta)\sum \limits_{k=j+1}^\infty r_k.
 \end{align*}
 This is the postulated statement for $j$ and inductively it holds for all $j \in \N_0$.
 \item $\tau_i$ converges on $S_0 \cap \overline{ B_{r'}(x)}$ uniformly to a continuous function $\tau$.\\
 For $y \in S_0\cap \overline{ B_{r'}(x)}$ and $i \in \N$ we get
 \begin{align*}
  \vert \tau_i(y)-\tau_{i-1}(y) \vert &=\vert \sigma_i( \tau_{i-1}(y)) - \tau_{i-1}(y) \vert\\
  &\leq \dist\left(\tau_{i-1}(y),\Sigma_x\right) + (36C_1(m) +24)r_i\delta.
 \end{align*}
 If $i=1$ then 
  \[ \dist\left(\tau_0(y), \Sigma_x\right) \leq r\delta <(36C_1(m)+24)r_0 \delta\]
 and for $i>1$ we get
  \[ \dist\left(\tau_{j-1}(y), \Sigma_x\right) = \dist\left( \sigma_{i-1}( \tau_{i-2}(y)),\Sigma_x\right) \leq (36C_1(m)+24) r_{i-1}\delta, \]
 because of $\tau_{i-2}(y)\in  V_{i-1}$. Using $r_i= \frac 1 4 r_{i-1}$ yields
  \[ \vert \tau_i(y) - \tau_{i-1}(y) \vert \leq \frac 5 4 \left( 36C_1(m)+24 \right) r_{i-1}\delta \Foa i \in \N. \]
 Let $k,j \in \N_0$ then
 \begin{align*}
  \vert \tau_{j+k}(y)- \tau_j(y)\vert &\leq \sum_{i=1}^k \vert \tau_{j+i}(y)-\tau_{j+i-1}(y) \vert \\
  &\leq \frac 5 4 \left( 36C_1(m)+24 \right) \delta \sum_{i=1}^k r_{j+i-1}\\
  &= \frac 5 4 \left( 36C_1(m)+24 \right) \delta r_j \sum_{i=0}^{k-1} 4^{-i}\\
  &\xrightarrow [j \to \infty ]{} 0.
 \end{align*}
 This is independent of $y \in S_0\cap \overline{ B_{r'}(x)}$ and implies the uniform convergence of $\tau_i$ to a function $\tau$. All $\tau_i$ are continuous as compositions of continuous 
 functions and therefore $\tau$ is as well.
 \item $\vert \tau(y)- y \vert < Cr\delta$ and $\tau(S_0\cap \overline{ B_{r'}(x)}) \subset \Sigma_x$.\\
 We have $\tau(y)= \lim_{j\to \infty} \tau_j(y)$ for all $y \in S_0\cap \overline{ B_{r'}(x)}$. Therefore, for all $\varepsilon>0$ there exists a $J=J(\varepsilon) \in \N$ with
  \[ \vert \tau(y)-\tau_j(y)\vert < \varepsilon \Foa j \geq J \AND y \in S_0 \cap \overline{ B_{r'}(x)}. \]
  For $k \in \N_0$ there is a $j > \max \{ k,J\}$ with
 \begin{align*}
  \vert \tau(y)-\tau_k(y)\vert &<\varepsilon + \sum_{i=k}^{j-1} \vert \tau_{i+1}(y) - \tau_i(y) \vert \\
  &\leq \varepsilon + \sum_{i=k}^{\infty} \vert \tau_{i+1}(y) - \tau_i(y) \vert.
 \end{align*}
 The limit $\varepsilon \to 0$ yields
 \begin{align*}
  \vert \tau(y)-\tau_k(y) \vert &\leq \sum_{i=k}^{\infty} \vert \tau_{i+1}(y) - \tau_i(y) \vert\\
  &\leq \frac 5 4 \left( 36C_1(m) +24\right) \delta r_k \cdot\sum_{i=0}^\infty 4 ^{-i}\\
  &= \frac 5 3 \left( 36C_1(m) +24\right) \delta r_k .
 \end{align*}
 Especially for $k=0$ we get
  \[ \vert\tau(y)-y\vert \leq \frac 5 3 \left( 36C_1(m) +24\right) \delta r_0 < \frac 5{144} r. \]
 We have $\tau_j(y)\in V_{j+1}$ for all $j \in \N_0$ and therefore there is a $w_j \in \Sigma_x$ with
  \[\vert\tau_j(y)-w_j\vert < r_{j+1} \Foa j \in N_0.\]
 This leads to
 \begin{align*}
  \dist\left(\tau(y),\Sigma_x\right) &\leq \vert \tau(y)- \tau_j(y) \vert + \vert \tau_j(y) - w_j\vert \\
  &\leq \frac 5 3 \left( 36C_1(m) +24\right) \delta r_j + r_{j+1}\\
  &\xrightarrow[j\to \infty]{}0,
 \end{align*}
 which implies $\tau(S_0\cap \overline{ B_{r'}(x)}) \subset \Sigma_x$ and finishes the proof.
 \end{enumerate}
\end{proof}

\begin{proof}[\textbf{Proof of Lemma \ref{lem:projectionsurj}}]
 Assume there exists a $\xi \in (x+L) \cap B_{\frac r 4}(x)$ such that $\pi_{x+L}(y) \neq {\xi}$ for all $y \in \Sigma \cap B_{\frac r 2}(x)$.
 Using Lemma \ref{lem:reifenbergfunction} leads to a continuous function $\tau \colon (x+L) \cap \overline{B_{\frac{15}{16}r}(x)} \to \Sigma \cap \overline{B_r(x)}$ with
  \[ \vert \tau(y)-y\vert \leq \frac 5 {144} r. \]
 Then for all $z \in (x+L) \cap \overline{B_{\frac r 3}(x)}$ we get
  \[ \vert \tau(z) - x \vert \leq \vert \tau(z)-z\vert + \vert z-x \vert \leq \frac 5 {144} r + \frac 1 3 r < \frac 1 2 r. \]
 Therefore, 
  \[ \pi_{x+L}\left(\tau(z) \right) \neq \xi \Foa z \in (x+L) \cap \overline {B_{\frac r 3}(x)}. \]
 Let $h \colon (x+L) \setminus \{\xi\} \to (x+L) \cap \partial B_{\frac r {12}}(\xi)$ be defined by
  \[ h(z) := \xi + \frac r {12} \cdot \frac {z - \xi} {\vert z - \xi \vert }. \]
 $h$ is a continuous projection of $(x+L) \setminus\{\xi\}$ onto $(x+L)\cap \partial B_{\frac r{12}}(\xi)$. Define
  \[ \varphi := h \circ \pi_{x+L} \circ \tau \colon (x+L) \cap B_{\frac r{12}}(\xi) \to (x+L)\cap \partial B_{\frac r {12}}(\xi). \]
 Note that $B_{\frac r {12}}(\xi) \subset B_{\frac r 3}(x)$, then we have $\xi \not\in \pi_{x+L}\circ \tau((x+L) \cap B_{\frac r{12}}(\xi)$ and $\varphi$ is continuous and well-defined.\\
 For $z \in (x+L) \cap \partial B_{\frac r {12}}(\xi)$ we get
 \begin{align*}
  \left \vert \pi_{x+L} \left( \tau(z) \right) -z\right\vert &= \left \vert \pi_{x+L}\left ( \tau(z)-z\right) \right\vert\\
  &\leq \vert \tau(z)-z\vert \\
  &\leq \frac 5 {144} r.
 \end{align*}
 Moreover,
 \begin{align*}
  \left \vert h\left(\pi_{x+L}(\tau(z))\right)-  \pi_{x+L}(\tau(z))\right \vert &= \dist\left(\pi_{x+L}(\tau(z)), \partial B_{\frac r{12}}(\xi)\right)\\
  &\leq \left \vert \pi_{x+L}(\tau(z))-z \right \vert\\
  &\leq \frac 5 {144} r,
 \end{align*}
 which implies 
  \[ \vert \varphi(z)-z \vert \leq \frac {10}{144}r \Foa z \in (x+L) \cap \partial B_{\frac r {12}}(\xi).\]
 Define $\tilde \varphi \colon L \cap \overline {B_1(0)} \to L \cap \partial B_1(0)$ by
  \[ \tilde \varphi (z) := \frac {12}r \left( \varphi \left( \frac r {12} z + \xi \right) - \xi \right). \]
 The continuity of $\varphi$ implies that $\tilde \varphi$ is also continuous and for $z\in L \cap \overline{B_1(0)}$ we get $\tilde z := \frac r{12}z+\xi \in(x+L) \cap \partial B_{\frac r {12}}(\xi)$, which leads to
  \[ \vert \tilde \varphi(z) - z \vert= \frac {12}r \vert \varphi(\tilde z) - \tilde z\vert \leq \frac {12} r \cdot \frac {10} {144} \cdot r = \frac {10} {12} <1 \Foa z \in L \cap \partial B_1(0).  \]
 But this implies that
  \[ H \colon L \cap \partial B_1(0) \times [0,1] \cong \S^{m-1} \times [0,1] \to L \cap \partial B_1(0) \cong \S^{m-1}, \]
  \[ H(z,t) := \frac {(1-t) \tilde \varphi_{\mid \S^{m-1}}(z) + tz}{\vert (1-t) \tilde \varphi_{\mid \S^{m-1}}(z) + tz \vert} \]
 is a homotopy between $id_{\S^{m-1}}$ and $\tilde \varphi_{\mid \S^{m-1}}$. The homotopy equivalence of the degree of a map (see \cite[5.1.6 a]{hirsch_1988}) leads to
  \[ \operatorname{deg}(\tilde \varphi_{\mid \S^{m-1}} ) = \operatorname{deg}(id_{\S^{m-1}}) = 1. \]
 This is a contradiction to the continuous extension $\tilde \varphi$ of $\tilde \varphi_{\mid \S^{m-1}}$ on $\overline{B^m_1(0)}$, because this would by \cite[5.1.6 b]{hirsch_1988} imply
  \[\operatorname{deg}(\tilde \varphi_{\mid \S^{m-1}} ) = 0. \]
 Therefore, the assumed $\xi$ can not exist.
\end{proof}

\end{appendix}

%%%%%%%%%%%%%%% Literatur %%%%%%%%%%%%%%%%%%%
 \bibliography{Literatur}
 \bibliographystyle{acm}
 %\input{Immersionen_Version2.bbl}
%%%%%%%%%%%%%%%%%%%%%%%%%%%%%%%%%%%%%%%%%%%%%%%%%%%%%

\end{document}